\documentclass[a4wide,11pt]{article}
\usepackage[DIV=11]{typearea}
\usepackage{amsmath,amssymb,amsthm}
\usepackage[dvipsnames]{xcolor}
\RequirePackage{aliascnt}
\usepackage[unicode, final, colorlinks=true]{hyperref}
\usepackage{textgreek}
\usepackage{enumitem, booktabs}
\usepackage{xcolor}
\usepackage{siunitx}
\usepackage{mathtools}
\usepackage{tikz, pgfplots}
\usepackage{romanbar}
\usepackage{caption}
\usepackage{float}
\usepackage{bm}
\usepackage[capitalise]{cleveref}
\usepackage[ruled,vlined]{algorithm2e}
\usepackage{subcaption}
\usepackage{placeins}
\usepackage{cite}

\hypersetup{citecolor=teal, linkcolor=BrickRed, urlcolor=NavyBlue}

\pgfplotsset{compat=1.17}
\usepgfplotslibrary{groupplots}
\usepgfplotslibrary{fillbetween}
\usetikzlibrary{backgrounds}
\usetikzlibrary{arrows.meta}
\pgfplotsset{plot coordinates/math parser=false}
\usetikzlibrary{external}
\usepgfplotslibrary{patchplots}
\usetikzlibrary{shapes.geometric}
\usetikzlibrary{backgrounds}
\usetikzlibrary{intersections}
\usetikzlibrary{spy}

\newcommand{\R}{ {\mathbb{R}}}

\newcommand{\e}{ {\varepsilon}}
\newcommand{\g}{ {\gamma}}
\renewcommand{\t}{ {\theta}}
\renewcommand{\r}{ {r}}
\newcommand{\ueb}{ {\bar{u}^{\varepsilon}}}
\newcommand{\ue}{ {u^{\varepsilon}}}
\newcommand{\uen}{ {u_0^{\varepsilon}}}
\newcommand{\uebn}{ {\bar{u}_0^{\varepsilon}}}
\newcommand{\un}{ {{u}^{0}}}
\newcommand{\unn}{ {{u}_0^{0}}}
\newcommand{\opl}{ {+}}
\newcommand{\omin}{ {-}}
\newcommand{\uebm}{ {{u}^{\e -}}}
\newcommand{\uebp}{ {{u}^{\e +}}}
\newcommand{\rotmat}{ {\begin{pmatrix} 0 &-1 \\ 1 &0 \end{pmatrix}}}
\newcommand{\D}{ {D}}

\newcommand{\gn}{ {{\gamma_h^0}}}
\newcommand{\nn}{ {{n_h^0}}}
\newcommand{\nun}{ {u_h^0}}
\newcommand{\nue}{ {\bar{u}_h^\e}}
\newcommand{\nueob}{ {u_h^\e}}
\newcommand{\nv}{ {v_h}}
\newcommand{\nr}{ {r_h}}
\newcommand{\sn}{ \sigma}

\providecommand{\keywords}[1]{\textit{Keywords:} #1}
\providecommand{\msc}[1]{\textit{2020 MSC:} #1}

\newcommand\newsharedtheorem[2]{
    \newaliascnt{#1}{theorem}
    \newtheorem{#1}[#1]{#2}
    \aliascntresetthe{#1}}

\theoremstyle{plain}
\newtheorem{theorem}{Theorem}[section]
\newsharedtheorem{lemma}{Lemma}
\newsharedtheorem{corollary}{Corollary}

\theoremstyle{definition}
\newsharedtheorem{definition}{Definition}
\newsharedtheorem{proposition}{Proposition}
\newsharedtheorem{remark}{Remark}
\newsharedtheorem{hypothesis}{Hypothesis}
\newsharedtheorem{example}{Example}
\newsharedtheorem{assumption}{Assumption}

\begin{document}
\date{
  \small
  Institute of Geometry and Applied Mathematics,\\ RWTH Aachen University, Im S\"usterfeld 2,\\ 52072 Aachen, Germany\\
   \smallskip
   September, 2026
  }

\title{Sensitivity Calculus and its Numerical Implementation for Multi-D Hyperbolic Balance Laws}
\author{Olivia Dre\ss en\thanks{\noindent Corresponding author. E-Mail: dressen@eddy.rwth-aachen.de \\
Contributing authors: herty@igpm.rwth-aachen.de; kolb@eddy.rwth-aachen.de; mueller@igpm.rwth-aachen.de} \and Michael Herty \and Adrian Kolb \and Siegfried M\"uller}

\maketitle

\begin{abstract}
We investigate the sensitivity of solutions to multi-dimensional scalar balance laws with respect to perturbations of the initial data analytically and numerically. Reliable first-order sensitivity information is essential in gradient-based optimization and inverse problems constrained by hyperbolic balance laws. In hyperbolic problems, such perturbations affect both the smooth components of the solution and the locations of shocks. Consequently, classical difference quotients of the solution operator generally fail to converge in $L^1$, even in one space dimension. 
Although generalized tangent-vector techniques have been developed for one-dimensional problems, extending these ideas to multiple space dimensions is considerably more challenging because it requires a geometric description of hypersurfaces of discontinuity.

We represent perturbations of hypersurfaces of discontinuity by normal displacements and account for the induced variation of the normal direction. This representation yields evolution equations for the components of a generalized tangent vector. Based on this calculus, we develop a numerical method for computing first-order variations with respect to the initial data. Numerical experiments in two space dimensions confirm the expected first-order accuracy of the resulting approximation.

  \bigskip
  \noindent \keywords{scalar multi-dimensional balance law, shock sensitivity, generalized tangent vectors, numerical schemes}\\
  \msc{ 35L65, 35L67, 65K10, 49K40, 49K20}
  \end{abstract}

\newpage
\section{Introduction}\label{sec: Intro}

The study of how to solve partial differential equations naturally raises questions about their sensitivity with respect to the data and parameters of the problem. 
We study the sensitivity of solutions to multi-dimensional hyperbolic equations with respect to perturbations of the initial data.
Such sensitivities are particularly relevant in optimization problems constrained by partial differential equations, where first-order variations provide the basis for necessary optimality conditions and, for example, gradient-based numerical methods.

In this paper, we study solutions to scalar equations in multiple space dimensions, possibly including source terms, of the form
\begin{equation*}
    \partial_t u + \nabla_x \cdot f(u) = S(u),
    \qquad x\in\mathbb{R}^d,\quad t>0,
\end{equation*}
together with an initial condition $u(\cdot,0)=u_0(\cdot)$. A fundamental difficulty in the analysis of such equations is the formation of shock discontinuities, which may occur in finite time even for smooth initial data, see, for instance, \cite[Section 3.3]{lV}. Once formed, such discontinuities propagate in time and are not necessarily smoothed out by the dynamics. As a consequence, the solution operator cannot, in general, be linearized in the classical sense. 

This difficulty is already present in the one-dimensional scalar case. In general, the evolution operator
$
\mathcal{S}_t : u_0 \mapsto u(\cdot,t)=\mathcal{S}_t u_0
$
is not differentiable in $L^1$, see \cite[Example 1]{osti_482447}. The reason is that perturbations of the initial data not only change the values of the solution in smooth regions, but may also shift the location of discontinuities. Difference quotients of such shifted jumps fail to converge in $L^1$.

In the \emph{spatially one-dimensional} case, this observation has led to the theory of tangent vectors or shift differentiability for solutions. Analytical and numerical works in this direction include, among many others, \cite{BandaHerty2012aa,BardosPironneau2002aa,Bianchini2000aa,BressanGuerra1995aa,BressanLewicka1999aa,osti_482447,BressanMarson1995aa,CastroPalaciosZuazua2008aa,GilesUlbrich2010aa,JamesSepulveda1999,PierceGiles2004aa,Giles1996aa,GilesSueli2002aa,GugatHertyKlar2006aa,Ulb02,Herty_2021,HZ24,PU15,SU21,AU26,Ulb03}. A key contribution is the tangent vector calculus introduced for spatially one-dimensional systems of balance laws in \cite[Theorems 2.2 and 2.3]{osti_482447}, see also \cite{BressanMarson1995aa}. In this framework, first-order perturbations are represented by variations of the smooth parts of the solution together with shifts of points of discontinuity. This point of view was further developed to prove continuous dependence of the solution operator \cite{BressanCrastaPiccoli2000aa} and extended to $BV$ initial data \cite{BressanGuerra1995aa,Bianchini2000aa}. Furthermore, it led to the notion of shift-differentiability \cite{Bianchini2000aa} as well as to an adjoint calculus \cite{BressanShen2007aa}. For scalar spatially one-dimensional problems, weaker assumptions on the initial data and related optimal control questions have been studied in \cite{Ulb03,CastroPalaciosZuazua2008aa}. The connection with weak formulations is discussed, for example, for the Burgers equation in \cite{BardosPironneau2002aa}. 
The notion of Fr\'echet differentiability has been considered for example in \cite{Ulb02,SU21,AU26,PU15,Ulb03}.
Progress has also been made in computing tangent vectors. For example, \cite{Herty_2021} presents an algorithmic-differentiation framework, whereas \cite{HZ24} treats the discontinuity curve as an interface.
\\

The \emph{spatially multi-dimensional} case is considerably more challenging. While a one-dimensional shock location can be described by a point and its perturbation by a scalar shift, shock locations in multiple spatial dimensions are hypersurfaces. Their first-order variations, therefore, have geometric features such as the displacement of the curve, but also changes of its normal direction.

To the best of our knowledge, the work \cite{LZ16} is the closest contribution in this direction for two spatial dimensions. There, a tangent vector approach is developed for two-dimensional scalar conservation laws and applied in the context of an optimization strategy. The analysis focuses on shock curves without boundary. The present paper complements this approach while remaining consistent with the setting in \cite{LZ16}. We derive a sensitivity calculus from a different perspective and generalize the results, extending the analysis to open shock hypersurfaces, equations with source terms, and more than two space dimensions. By open shock hypersurfaces, we mean shock hypersurfaces whose boundaries may lie in the interior of the domain, rather than being required to connect to the domain boundary. Furthermore, we present a numerical method.

The aim of this work is twofold. First, we derive evolution equations for tangent vector components describing both the perturbation of the smooth solution components and the normal displacement of the shock surface, including the induced variation of the normal direction. Second, we design and implement a numerical method based on this calculus. The numerical experiments compare first-order approximations obtained by the tangent vector and reference solutions for perturbed initial data to confirm the expected first-order accuracy.

This paper is organized as follows. 
In \cref{sec: problem description} we describe the considered class of multi-dimensional scalar balance laws and specify the assumptions. 
\cref{sec: gen tangent vec} introduces the notion of generalized tangent vectors. We first give a formal derivation from perturbations of the initial data and then discuss the evolution of regular variations, which describe the first-order approximations of perturbed solutions.
In \cref{sec: Implementation} we present the numerical methodology for the two-dimensional case. This includes the computation of the unperturbed solution to the balance law, the tangent vector and auxiliary variables.
The numerical results are reported in \cref{sec: num results}. We first discuss the expected accuracy and then consider two test cases, including an open shock curve
 with endpoints in the domain, and a shock curve that meets the domain boundary and separates the domain into two subdomains.
Finally, \cref{sec: conclusion} summarizes the main findings and discusses possible directions for future work.

\FloatBarrier
\section{Description of the Problem}\label{sec: problem description}
For a bounded domain ${\Omega} \subset \R^d$ for $d\ge 2$, we consider the initial value problem for a scalar balance law given by
\begin{equation}\label{eq: IVP}
    \begin{array}{rclcc}
    \partial_t \un + \nabla_x \cdot f(\un) &= &S(\un) & \text{for }x \in \Omega, & t>0 \\
    \un\left( x,0 \right) &= &\unn(x) & \text{for }x \in \Omega.
    \end{array}
\end{equation}
with $\unn:\Omega \to \R$ and source $S \in C^2(\R,\R)$.
Assume $ f \in C^2(\mathbb{R},\mathbb{R}^d)$ and 
\begin{equation*}
    \unn \in \mathcal{U} \coloneqq \left\{u:\Omega \to \R \vert \text{$u$ measurable, } {TV}(u)\leq C, u \text{ piecewise Lipschitz continuous} \right\}
\end{equation*}
as well as suitable boundary conditions.
To investigate the sensitivity of the solution $\un$ with respect to perturbations in the initial data we introduce a perturbed problem.
The perturbed initial data is denoted by $\uen(x)$ and its corresponding solution at time $t$ by $\ue(x,t)$. 
The perturbed initial value problem takes the form 
\begin{equation}\label{eq: perturbed conslaw}
    \begin{array}{rclcc}
    \partial_t \ue + \nabla_x \cdot f(\ue)& = & S(\ue) & \text{for } x \in \Omega, & t>0 \\
    \ue(x,0) & = &\uen(x) & \text{for } x \in \Omega.
    \end{array}
\end{equation}
For both initial-value problems, we consider entropy solutions. The discontinuity hypersurfaces considered below are assumed to be entropy-admissible shocks satisfying the Rankine-Hugoniot condition.
We collect the remaining assumptions below.
\begin{assumption}\label{assumptions}
    ~
    \begin{itemize}
        \item[(a)] Without loss of generality, we assume that the initial datum $\unn$ has only one discontinuity along the hypersurface parametrized by $\g_0^0$ and the perturbed initial data $\uen$ also contains only one discontinuity parametrized by $\g_0^\e$. Furthermore, for $t>0$, we denote the parametrizations by $\g^0=\g^0(\cdot,t)$ or $\g^\e=\g^\e(\cdot,t)$, respectively. We also assume that for all $t \in [0,T)$ for some $T>0$, no additional discontinuities arise in $\un$ and $\ue$. 
        \item[(b)] We choose the spatial parameter domain of $\g^0$ and $\g^\e$ as $\D \coloneqq [0,1]^{d-1}$, i.e. \\ $\g_0^0(\cdot),\g_0^\e(\cdot),\g^0(\cdot,t),\g^\e(\cdot,t): \D \to \R^d$.
        \item[(c)] Additionally, we assume the parametrization to be injective, i.e., there are no self-intersections. Furthermore, we assume $\operatorname{rank}(D_\t \g^0) =d-1$. Lastly, we only consider $\g_0^0 \in C^2(\D,\R^d)$.
        \item[(d)] We assume in the following that $\g^0$ remains $C^2$ for all $t \in [0,T)$ for some $T>0$.
    \end{itemize}
\end{assumption}

Concerning \cref{assumptions} (a): In principle, multiple hypersurfaces of discontinuities that do not intersect at a given time can also be considered one by one. Concerning the choice of the parameter domain $\D$ in (b), we note that for technical reasons concerning the function to be defined in \cref{eq: theta_stern}, we need the image of the parametrization $\g^0$ to be a compact submanifold. This is why we choose a compact set for $\D$. 
We distinguish between two types of hypersurfaces: on the one hand, hypersurfaces separating the domain $\Omega$ into two parts, on the other hand hypersurfaces that do not touch $\partial \Omega$, i.e. end in the domain. 

For hypersurfaces that separate $\Omega$ into two parts, it is common to choose a compact parameter domain. This has been done, for example, in \cite{Maj83}.

For the second class of hypersurfaces, the choice of the parameter domain might not be natural, because the discontinuity vanishes on parts of the boundary, where $\g^0(\t,t)$ involves smooth data of $\un$. For the first-order approximation of the perturbed solution, only the points of $\g^0$ where the jump in $\un$ is non-zero affect the calculus. Therefore, without loss of generality, we include these parts of the boundary in the following theory also in this case. 

\cref{assumptions} (c) is needed for technical reasons, especially in \cref{prop: well defined,prop: well defined time}.
Finally, \cref{assumptions} (d) holds true, provided some stronger regularity assumptions on the initial value problem are imposed:

\begin{proposition}\label{prop: gamma c2}
    ~\\
    Suppose that the preceding assumptions hold. In addition, assume that $\g_0^0 \in C^2(\D,\R^d)$, $u_0^0 \in C^2(\Omega\setminus \operatorname{Im}(\g_0^0),\R)$ away from the discontinuity and $f\in C^3(\R,\R^d)$. Assume that $\un$ does not develop other discontinuities, i.e. $\g^0$ describes the only discontinuity in $\un$ at any considered time. Then, $\g^0(\cdot,t)$ stays $C^2$-regular locally in time. 
\end{proposition}
One can prove this proposition using the classical theory of characteristics, see e.g. \cite[Section 3.3]{PDE2015}.
We emphasize that the result of \cref{prop: gamma c2} can, in practice, also be obtained under weaker regularity assumptions, but it shows that \cref{assumptions} (d) can actually be fulfilled.

\FloatBarrier
\section{Generalized Tangent Vectors}\label{sec: gen tangent vec}

It was shown in \cite{osti_482447} that the evolution operator $\mathcal{S}_t : u_0(\cdot ) \to  u(t, \cdot ) = \mathcal{S}_t u_0(\cdot )$, of the balance law, is in general non-differentiable in $L^1$,
even in the one-dimensional scalar case. This means that the limit $\lim\limits_{h \to 0} \left( u^{\e+h}-\ue \right)/h$ does not necessarily define a function in $L^1$. For this reason, a calculus for generalized tangent vectors in the case of one-dimensional systems of balance laws was introduced, see e.g. \cite{PU15,Ulb02,SU21,AU26,Ulb03,osti_482447}.

The section is divided into two subsections. First, in \cref{sec: 3.1} we discuss the approximation of the perturbed initial data. Then, in \cref{sec: 3.2}, we introduce the evolution equations to propagate the components of the approximation of the perturbed solution in time to obtain an approximation also for $t>0$.
\FloatBarrier
\subsection{Formal Derivation of Perturbation of the Initial Data}\label{sec: 3.1}

In this section, we define the pair of functions $\left( v,\r \right)$ with $v \in L^1\left( \R^d,\R \right)$ and $\r \in C^1(D,\R) $ describing the infinitesimal $L^1$-displacement and the infinitesimal displacement of the shock position, respectively.
The space of these pairs is consequently given by $T_{\un} \coloneqq  L^1\left( \R^d,\R \right) \times C^1\left( \D,\R \right)$.
We note that $r$ can also be considered in the weak sense, i.e. $\r \in W^{1,1}(D,\R) $, as we will also do for the numerical experiments in \cref{sec: Implementation} and \cref{sec: num results}.

This pair of functions $\left( v,\r \right)$ refers to the generalized tangent vector in one spatial dimension, introduced in \cite{osti_482447}. Even though for the generalized tangent vector in one dimension, several properties have been proven that are not yet proven for our multi-dimensional approach, we refer to $\left( v,\r \right)$ as generalized tangent vector or just tangent vector in the following.

To approximate the perturbed initial data, we want to construct a first-order variation, i.e. for perturbed initial data $\uen$ and its approximation $\uebn$, we aim for the relation \\\mbox{$ \| \uen(\cdot) - \uebn(\cdot) \|_{L^1(\Omega)} = o(\e)$}.
In the following, we use the notation $\doteq$ to denote equality up to order $o(\e)$, i.e.,
\begin{equation*}
    a^\e(\cdot) \doteq b^\e(\cdot) \Leftrightarrow a^\e(\cdot) = b^\e(\cdot) + o(\e),
\end{equation*}
with respect to the corresponding functionspace. The space is e.g. given by $L^1(\Omega)$ for $\ue$.

In the following we introduce some notations. With regard to investigations below, they already incorporate the time variable.\\
We define the jump in $\ue$ along $\g^\e$ for $\e\ge 0$ as
\begin{equation}\label{eq: jump definition}
\begin{split}
    [\ue]\left( \t,t \right)&\coloneq \lim_{\delta \to 0} \ue\left( \g^\e\left( \t,t \right) + \delta n^\e\left( \t,t \right),t \right) - \ue\left( \g^\e\left( \t,t\right) - \delta n^\e\left( \t,t \right) ,t \right)\\
    &\equiv \ue\left( \g^\e\left( \t,t \right) \opl ,t\right) - \ue\left( \g^\e\left( \t,t \right) \omin ,t\right),
\end{split}
\end{equation}
where $n^\e$ denotes the spatial normal of the hypersurface parametrized by $\g^\e$. Here, $n^0$ corresponds to $\g^0$ and $n^\e$ corresponds to $\g^\e$. 
Without loss of generality, the normal is given by
\begin{equation}\label{eq:definition normal}
    n^\e\left( \t,t \right) = \frac{1}{\sqrt{\operatorname{det}\left( (D_\t\g^\e)^T (D_\t\g^\e) \right)}}
    \bigg( \left( -1 \right)^{i+d} \operatorname{det}\left( (D_\t\g^\e)_{\hat{i}} \right) \bigg)_{i=1}^d,
\end{equation}
for $\e\ge 0$ where $( D_\t\g^\e )_{\hat{i}} \in \R^{(d-1)\times(d-1)}$ describes the submatrix of $D_\t\g^\e$ that results from $D_\t\g^\e \in \R^{d\times(d-1)}$ by removing the $i$-th row.
A more detailed discussion on this can be found in the Appendix \ref{sec: normals Rd}. 
In the two-dimensional case we will refer to the hypersurface parametrized by $\g^\e$ as \textit{curve}. Later, in \cref{sec: 2d reduction}, we discuss the simplifications applying to the two-dimensional case. 
We approximate the perturbed shock surface with scalar displacements $\r$ in the normal direction $n^0$ of $\g^0$:
\begin{equation}\label{eq: gamma eps expansion}
    \g^\e(\t,t) \doteq \g^0(\t,t) + \e \r(\t,t) n^0(\t,t).
\end{equation}
For an illustration in two dimensions, we refer to \cref{fig:illustration_shockpos_shift}, where an unperturbed initial curve $\g_0^0$ (green) and an example of a perturbed initial curve $\g_0^\e$ (blue) are shown. Along the normal of $\g_0^0$, we approximate the distance between the two curves with a first-order variation $ \r_0$. 
The endpoints of the curves coincide in \cref{fig:illustration_shockpos_shift}, but the construction also applies when they do not.

For the normal to the perturbed shock position parametrized by $\g^\e$, we introduce the linear expansion 
\begin{equation}\label{eq: neps expansion}
    n^\e\left( \t,t \right) \doteq n^0\left( \t,t \right) + \e \eta\left( \t,t \right).
\end{equation}
The  first-order term is then calculated by
\begin{equation}\label{eq: def eta}
    \eta(\t,t) = \frac{1}{\sqrt{\operatorname{det}\left( (D_\t\g^0(\t,t))^T (D_\t\g^0(\t,t)) \right)}} H(\t,t) \left( \nabla_\t r(\t,t) \right)
\end{equation}
where $H \in \R^{d \times \left( d-1 \right)}$ is given as
\begin{equation*}
    H \coloneq \Bigg( \left( -1 \right)^{i+d} \left( \operatorname{adj}\left(( D_\t \g^0(\t,t)  )_{\hat{i}}\right) (n^0(\t,t))_{\hat{i}} \right)^T  \Bigg)_{i=1}^{d} 
\end{equation*}

For details on the derivation, we refer again to the Appendix \ref{sec: normals Rd}.

To approximate the shift of the discontinuity in the perturbed initial data, we introduce the projection of a point in the neighborhood of $\g^0$ onto $\g^0$:

\begin{equation}\label{eq: theta_stern}
    \t^*:U_t \times \R_{\ge 0} \to \D, \,\, (y,t) \mapsto\underset{\t \in \D}{\operatorname{argmin}} \|y-\g^0(\t,t) \|,
\end{equation}
for $U_t$ a time-dependent neighborhood of $\g^0$. The map $\t^*$ gives the corresponding parameter to the closest point on the hypersurface $\operatorname{Im}(\g^0(\cdot,t))$, which will be used in the approximation of $\uen$ to assign a jump value in the unperturbed solution to a given point in $\Omega$. 
It is illustrated in \cref{fig:illustration_shockpos_shift} for the initial data at a fixed $\t$.
\begin{proposition}[Well-definedness of $\t^*$]\label{prop: well defined}
    ~\\
    Let $t>0$ be fixed, $\D$ compact and $\g^0(\cdot,t):\D \to \R^d$ an injective $C^2$-function with $\operatorname{rank}(D_\t \g^0) =d-1$. Then, the map $\t^*$ is well-defined in a neighborhood $U_t$ of the hypersurface $\operatorname{Im}(\g^0(\cdot,t))$.
\end{proposition}
\begin{proof}
    From differential geometry, see e.g., \cite{Foo84}, a compact $C^2$-submanifold has a neighborhood with the unique nearest point property. The image of $\g^0$ is a $C^2$-submanifold, since $\g^0$ is injective and $C^2$ and it holds $\operatorname{rank}(D_\t \g^0) =d-1$ by \cref{assumptions}, see e.g. \cite{Hir76}. Moreover, its image is compact since the domain $\D$ is compact. Therefore, there exists a neighborhood $U_t$ in which the unique nearest point property is fulfilled. Since $\g^0(\cdot,t)$ is injective, the nearest point corresponds to a unique parameter in the domain $\D$.
\end{proof}
The well-definedness is restricted to the neighborhood in which the unique nearest point property holds. Therefore, we apply $\t^*$ only within this neighborhood.

Let $\e$ be sufficiently small and $\left( v_0,\r_0  \right) \in T_{\unn}$. Then, we approximate the perturbed initial data $\uen$ by 
\begin{equation}\label{eq:approx}
    \uebn(x)  \coloneqq  \unn(x) + \e v_0(x) - \operatorname{sign}(\r_0\left( \t^*(x,0) \right)) [\unn]\left( \t^*(x,0) \right)\chi_{M_0}(x) \doteq \uen(x),
\end{equation}
where $M_0 \coloneqq  \left\{\g_0^0(\t) + \e \r_0( \t) n^0(\t,0) \alpha :\alpha \in  [0,1], \, \t \in \D \right\}$.\\

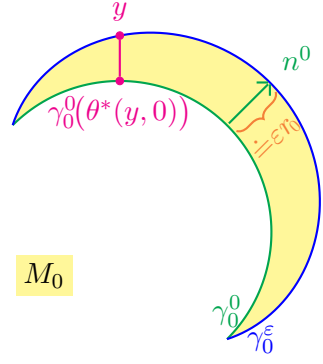
\begin{figure}
      \centering
  \scalebox{1}{
    \begin{tikzpicture}

\fill[yellow!50]
    plot[domain=-0.25*pi:0.75*pi, samples=100, variable=\t]
      ({2*cos(\t r)}, {2*sin(\t r)})
    --
    plot[domain=0.75*pi:-0.25*pi, samples=100, variable=\t]
      ({(2+0.8*(-4/(pi^2))*(\t+0.25*pi)*(\t-0.75*pi))*cos(\t r)},
       {(2+0.8*(-4/(pi^2))*(\t+0.25*pi)*(\t-0.75*pi))*sin(\t r)})
    -- cycle;

  \draw[thick,Green]
    plot[domain=-0.25*pi:0.75*pi, samples=100, variable=\t]
      ({2*cos(\t r)}, {2*sin(\t r)});

      \draw[thick,blue]
    plot[domain=-0.25*pi:0.75*pi, samples=100, variable=\t]
      ({(2+0.8*(-4/(pi^2))*(\t+0.25*pi)*(\t-0.75*pi))*cos(\t r)}, {(2+0.8*(-4/(pi^2))*(\t+0.25*pi)*(\t-0.75*pi))*sin(\t r)});

      \node[inner sep=0pt]
    (A) at ({2*cos(0.25*pi r)},{2*sin(0.25*pi r)}) {};

    \node[inner sep=0pt]
    (A_kl) at ({2*cos(0.25*pi r)+0.1},{2*sin(0.25*pi r)-0.1}) {};

    \node[inner sep=0pt,label=above right:\textcolor{Green}{$n^0$}] (B) at ({2*cos(0.25*pi r)+ cos(0.25*pi r)*0.8*(-4/(pi^2))*(0.25*pi+0.25*pi)*(0.25*pi-0.75*pi)},{2*sin(0.25*pi r)+sin(0.25*pi r)*0.8*(-4/(pi^2))*(0.25*pi+0.25*pi)*(0.25*pi-0.75*pi)}) {};

    \node[inner sep=0pt] (B_kl) at ({2*cos(0.25*pi r)+ cos(0.25*pi r)*0.8*(-4/(pi^2))*(0.25*pi+0.25*pi)*(0.25*pi-0.75*pi)+0.1},{2*sin(0.25*pi r)+sin(0.25*pi r)*0.8*(-4/(pi^2))*(0.25*pi+0.25*pi)*(0.25*pi-0.75*pi)-0.1}) {};

    \node[inner sep=0pt] (B2) at ({2*cos(0.25*pi r)+ cos(0.25*pi r)*0.8*(-4/(pi^2))*(0.25*pi+0.25*pi)*(0.25*pi-0.75*pi)-0.2},{2*sin(0.25*pi r)+sin(0.25*pi r)*0.8*(-4/(pi^2))*(0.25*pi+0.25*pi)*(0.25*pi-0.75*pi)}) {};

    \node[inner sep=0pt] (B3) at ({2*cos(0.25*pi r)+ cos(0.25*pi r)*0.8*(-4/(pi^2))*(0.25*pi+0.25*pi)*(0.25*pi-0.75*pi)},{2*sin(0.25*pi r)+sin(0.25*pi r)*0.8*(-4/(pi^2))*(0.25*pi+0.25*pi)*(0.25*pi-0.75*pi)-0.2}) {};

      \draw[thick,Green]
    (A) -- (B);

    \draw[thick,Green]
    (B) -- (B2);

    \draw[thick,Green]
    (B) -- (B3);

    \draw[decorate,decoration={brace,amplitude=4pt},
        Orange,thick] (B_kl) -- (A_kl);

        \node[inner sep=0pt,label={[label distance=2pt,rotate=45]below:\textcolor{Orange}{$\hspace{-0.1cm}\doteq \hspace{-0.1cm}\e r_{\scalebox{0.6}{\hspace{-0.1cm}$0$}}$}}] (B_kl) at ({2*cos(0.25*pi r)+ 0.5* cos(0.25*pi r)*0.8*(-4/(pi^2))*(0.25*pi+0.25*pi)*(0.25*pi-0.75*pi)+0.1},{2*sin(0.25*pi r)+0.5*sin(0.25*pi r)*0.8*(-4/(pi^2))*(0.25*pi+0.25*pi)*(0.25*pi-0.75*pi)-0.1}) {};

          \node[circle,draw,inner sep=1pt,fill,magenta,label={above:\textcolor{Magenta}{$y$}}] (y) at ({2*cos(0.5*pi r)- cos(0.5*pi r)*0.8*(-4/(pi^2))*(0.5*pi+0.25*pi)*(0.5*pi-0.75*pi)},{2*sin(0.5*pi r)+sin(0.5*pi r)*0.8*(-4/(pi^2))*(0.5*pi+0.25*pi)*(0.5*pi-0.75*pi)}) {};

          \node[circle,draw,inner sep=1pt,fill,magenta,label={below:\textcolor{Magenta}{$\g_0^0\!\bigl(\t^*(y,0)\bigr)$}}]
    (stern) at ({2*cos(0.5*pi r)},{2*sin(0.5*pi r)}) {};

    \draw[thick,Magenta]
    (y) -- (stern);

\node[inner sep=0pt,label={above:\textcolor{Green}{$\g_0^0$}}]
    (g1) at ({2*cos(-0.25*pi r)},{2*sin(-0.25*pi r)}) {};
    \node[inner sep=0pt,label={[label distance=3pt]right:\textcolor{blue}{$\g_0^\e$}}]
    (g2) at ({2*cos(-0.25*pi r)},{2*sin(-0.25*pi r)}) {};

    \node[inner sep=0pt,label={above:\colorbox{yellow!50}{$M_0$}}]
    (M) at ({-1},{-1}) {};

\end{tikzpicture}}
  \captionof{figure}{Illustration of the shift in the shock position for the initial data in two dimensions, following \cref{eq: gamma eps expansion}. The shock position described by $\g_0^0$ (green) is shifted along the normal by an amount of $\e \r_0$ (orange). This approximates the perturbed curve $\g_0^\e$, where the correspondence between the neighborhood $M_0$ (see \cref{eq:approx}) and the unperturbed shock curve is formalized by the projection $\t^*$ defined in \cref{eq: theta_stern}.}\label{fig:illustration_shockpos_shift}
\end{figure}

\FloatBarrier
In the definition of the set $M_0$, we use the line-segment parametrization between $\g_0^0(\t)$ and $\g_0^0(\t) + \e \r_0( \t) n^0(\t,0)\doteq \g_0^\e$. Since we consider the whole hypersurface $\operatorname{Im}(\g_0^0)$, we include the line-segment parametrization for every $\t \in \D$. In \cref{fig:illustration_shockpos_shift} we see an example for this in the two-dimensional setting.

Along each normal, we take the value of $[\unn]$ corresponding to the point on the curve from which the normal vector emanates. Therefore, we use the projection $\t^*$ which is the parameter corresponding to that point on the curve. 
As we see in \cref{eq:approx}, the projection $\t^*(y,0)$ is only relevant if $y \in M_0$. The neighborhood of the hypersurface in which the well-definedness is valid, see \cref{prop: well defined}, should include the set $M_0$. This is guaranteed for sufficiently small $\e$ but strongly depends on the geometry of the hypersurface parametrized by $\g^0$ and therefore the underlying problem. Moreover, the well-definedness is needed, to assign to every point in $M_0$ a unique point on the hypersurface parametrized by $\g_0^0$.

Finally, we need to know the sign of $\r_0$, to determine wether the shock should be shifted along the normal or in the opposite direction. Depending on the sign, we add or subtract the value $[\unn](\t^*(x,0))$.

The $L^1$ infinitesimal displacement for the solution is described by $\unn(x) + \e v_0(x)$ in \cref{eq:approx}. Note that the $L^1$ infinitesimal displacement is not illustrated in \cref{fig:illustration_shockpos_shift}.\\

We assume there exists a generalized tangent vector for the initial data, i.e., for $t=0$. 
We denote this by $\left( v_0,\r_0 \right)$. 
The main goal in the following section is to derive evolution equations that describe the time evolution of the tangent vector $\left( v,r \right)$.

\FloatBarrier
\subsection{Regular Variation}\label{sec: 3.2}

We assume that $\un \in  L^1\left( \R^d,\R \right)$ is piecewise Lipschitz continuous with one discontinuity surface $\operatorname{Im}(\g^0(\cdot,t))$. 
Consider $\Sigma_{\un}$ as the family of continuous paths $\Gamma: [0,\e_0] \to L^1\left( \R^d,\R \right)$ with $\Gamma(0)=\un$ and $\e_0$ possibly depending on $\Gamma$, such that $\t^*$ is well-defined. This notion is consistent with restrictions in the one-dimensional calculus presented in \cite{osti_482447}.
\begin{definition}[Regular variation]\label{def:reg vari}
    The space of generalized tangent vectors to a piecewise Lipschitz continuous function $\un$ with a discontinuity at $\operatorname{Im}(\g^0(\cdot,t))$ for $\t \in  \D$ is given by $T_{\un}  \coloneqq  L^1\left( \R^d,\R \right) \times  C^1\left( \D,\R \right)$.
    A continuous path $\Gamma \in  \Sigma_{\un}$ generates a tangent vector $\left( v,\r \right) \in  T_{\un}$ if 
    \begin{equation}
        \lim_{\e \to 0} \frac{1}{\e} \|\Gamma(\e) - \bar{\Gamma}(\e) \|_{L^1} =0
    \end{equation}
    for 
    \begin{equation}\label{eq: reg vari}
    \begin{split}
        \bar{\Gamma}(\e) \equiv \ueb(x,t)  \coloneqq  \un(x,t) + \e v(x,t) -\operatorname{sign}\left( \r\left( \t^*(x,t),t \right) \right) [\un](\t^*(x,t),t)\chi_{M(t)}(x),
    \end{split}
    \end{equation}
    with 
    \begin{equation}
        M(t)= \left\{\g^0(\t,t) + \e \r( \t,t) n^0(\t,t) \alpha : \alpha \in  [0,1], \, \t \in \D \right\}
    \end{equation}
    and $\t^*$ as defined in \cref{eq: theta_stern}.\\
    Let $\un$ be a piecewise Lipschitz continuous function with non-interacting discontinuities. 
    Then, a path $\Gamma \in  \Sigma_{\un}$ is a regular variation for $\un$ if additionally all functions $\Gamma(\e) = \ue$ are piecewise Lipschitz with non-interacting discontinuities and the location of the jump at $\operatorname{Im}(\g^\e(\cdot,t))$ depends continuously on $\e$.
\end{definition}
\cref{def:reg vari} ensures that the approximations resulting from the regular variation are of first-order.
Note that the set $M(t)$ depends on $t$, which needs to be taken into account for the well-definedness of $\t^*$:
\begin{proposition}[Well-definedness of $\t^*$ in time evolution]\label{prop: well defined time}
    ~\\
    Under the assumptions of \cref{prop: well defined} for $t\in [0,T_2)$ and assumption of the well-definedness of $\t^*$ on $M_0$, the well-definedness of $\t^*$ on $M(t)$ holds true locally in time for $t\in [0,T_1)$, $T_1\le T_2$. 
\end{proposition}
The last proposition follows from the continuity of the functions defining $M(t)$, where we discuss the continuity of $\r$ in \cref{prop: regularity sol}.\\
For an illustration of the expansion \cref{eq: reg vari} in two dimensions, we refer to \cref{fig:illustration_shockpos_shift} and the corresponding explanations that refer to the initial data. The only change that is made for \cref{eq: reg vari} is the time-dependency of the components.\\

The following theorem provides evolution equations for the tangent vector components. 

\begin{theorem}\label{theorem}
    Let $\un(\cdot ,\cdot )$ be a piecewise Lipschitz continuous solution to the initial value problem \cref{eq: IVP} with initial data $\un(x,0)=\unn(x)$ that is piecewise Lipschitz with only one shock hypersurface at $\operatorname{Im}(\g^0)$ with $\g^0(\cdot,t) \in C^2(\D,\R^d)$.
    Let $\left( v_0,\r_0 \right) \in  T_{\unn} $ be a tangent vector to $\unn$ generated by the regular variation $\Gamma_0$ with $\Gamma_0(\e) = \uen$.
    Let $\ue(x,t)$ be the solution to the initial value problem \cref{eq: perturbed conslaw} with perturbed initial data $\ue(x,0)=\uen(x)$. We assume that the regular variation for $\ue$ exists for the tangent vector $\left( v,r \right)$ for $t>0$.
    \\
    Then, the evolution equations for $\left( v,r \right)$ are given by the following initial value problems 
    \begin{equation}\label{eq:evolution v}
        \partial_t v + \nabla_x \cdot \left( f'(\un) v \right) =  S'(\un)  v, \,\, v(0,\cdot )=v_0(\cdot ) 
    \end{equation}
    \textbf{away} from the discontinuity of $\un$ and 
    \begin{equation}\label{eq:evolution r}
    \begin{split}
            {\partial_t \r(\t,t)} +  g(\t,t) \cdot (\nabla_\t\r(\t,t))  = \tilde{g}(\t,t)  {\r(\t,t)} + \hat{g}(\t,t),  \quad  {r}(\cdot,0 )= {r_0}(\cdot ),
    \end{split}
\end{equation}
    \textbf{along} the discontinuity $\g^0\left( \t,t \right)$ of $\un$.\\
    The flux and source terms are given by 
    \begin{equation}\label{eq:evolution r: functions}
    \begin{split}
        g& \coloneqq - \frac{1}{[\un] \sqrt{\operatorname{det}\left( (D_\t\g^0)^T (D_\t\g^0) \right)} } [f(\un)]^T H , \\
        \tilde{g}& \coloneqq  \frac{n^0}{[\un]^2} \cdot \Big( [\un]\left[ f'(\un) \left( \nabla_x \un \cdot n^0 \right) \right] - \left[ f(\un)\right] \left( [\nabla_x \un] \cdot n^0 \right) \Big)\\
        \hat{g}& \coloneqq  \frac{n^0}{[\un]^2} \cdot \left( [\un] [f'(\un) {v}] - [f(\un)] [{v}] \right).
    \end{split}
    \end{equation}
\end{theorem}
\cref{theorem} extends to finitely many non-intersecting shock hypersurfaces.
In this case, we have an evolution equation for the shift of each hypersurface. Furthermore, for the practical usage in the numerical discretization in \cref{sec: Implementation,sec: num results}, it is convenient to transform the evolution equation for $\r$ into a conservative form, see \cref{cor:conservative r eq}.

The solution of \cref{eq:evolution v} is understood as a broad solution, i.e., almost every point in space-time is the starting point of a characteristic determining the solution, see \cref{broad solution}.

Note that the evolution equation for the shock position shift \cref{eq:evolution r} is a scalar, spatially $(d-1)$-dimensional transport equation with source terms and in particular a partial differential equation instead of an ordinary differential equation as in the one-dimensional case presented in \cite{osti_482447}. However, it is consistent with the one-dimensional case. Furthermore, note that $\hat{g}$ depends on $v$, since shifts in the smooth parts of the solution also affect the shock position shift.

Under the Assumptions of \cref{prop: gamma c2} concerning the regularity of $\g^0$, we state the following regularity for $\left( v,\r \right)$.
\begin{proposition}[Existence and regularity of solutions to \cref{eq:evolution v,eq:evolution r}]\label{prop: regularity sol}
    ~\\
    Let the assumptions of \cref{prop: gamma c2} hold true. Let $v_0 \in C^1(\R^d\setminus \operatorname{Im}(\g_0^0),\R)$ away from the discontinuity and $\r_0 \in C^1(\R^{d-1},\R)$. Then, there exist solutions to \cref{eq:evolution v,eq:evolution r}, with $v \in C^1(\R^d\setminus \operatorname{Im}(\g^0(\cdot,t)) \times [0,T) ,\R)$ away from the discontinuity and $\r \in C^1(\R^{d-1} \times [0,T),\R)$ locally in time. 
\end{proposition}
This result follows again from the theory of characteristics, see e.g. \cite[Section 3.3]{PDE2015}. \\
Together with \cref{prop: well defined time}, we obtain from \cref{prop: regularity sol} the well-definedness of the regular variation as defined in \cref{def:reg vari} locally in time.

Finally, we present the proof for \cref{theorem}.
\begin{proof}[Proof of \cref{theorem}]
We derive evolution equations assuming that the expansion in the form of a regular variation for every $t>0$ and $\e \in [0,\e_0]$ is given. 

We start with the \textbf{evolution equation for $v$}. Away from the discontinuity of $\ue$, we have the expansion $\ue \doteq \un + \e v$. We assume that $\ue$ satisfies the initial value problem \cref{eq: perturbed conslaw}.
Up to terms of order $o(\e)$ we have
\begin{equation*}
    \begin{split}
        S(\un(x,t) + \e v(x,t))&=\partial_t \left( \un(x,t) + \e v(x,t) \right) + \nabla_x \cdot f(\un(x,t) + \e v(x,t)). 
    \end{split}
\end{equation*}
Using Taylor-expansion yields
\begin{equation*}
    \begin{split}
S(\un(x,t)) + \e S'(\un(x,t))v(x,t) \doteq &\partial_t \un(x,t) +\e \partial_t v(x,t)\\ 
 &+ \nabla_x \cdot \left( f(\un(x,t)) + \e f'(\un(x,t)) v(x,t)  \right).
    \end{split}
\end{equation*}
Collecting the first-order terms, we get 
\begin{equation}
    S'(\un(x,t))v(x,t) \doteq \partial_t v(x,t) +  \nabla_x \cdot \left( f'(\un(x,t)) v(x,t) \right).
\end{equation}
We continue with the \textbf{evolution equation for $\r$}:

If the shock vanishes on parts of the boundary of the hypersurface $\operatorname{Im}(\g^0(\cdot,t))$, we restrict ourselves to $\t \in \D \setminus E$, where $E$ describes the set of points where the shock is vanishing, i.e., $[\un](\t,t)=0$ for all $ \t \in E$. For $\t \in E$, the term in \cref{eq: reg vari} that includes $\r$ vanishes. 

First, we describe the shock speed using the Rankine-Hugoniot condition, see e.g. \cite{Daf16}.
With this, we obtain the following two equations for the shock speeds of the perturbed and unperturbed solution
\begin{equation}\label{eq: shockspeeds}
    \begin{split}
        s^\e = \frac{[f(\ue)]}{[\ue]} \cdot n^\e, \quad
        s^0 = \frac{[f(\un)]}{[\un]} \cdot n^0.
    \end{split}
\end{equation}
We expand all quantities and combine the expansions, such that all perturbed quantities are obtained in the form of a first-order approximation.

Using the expansion for $\g^\e$ in \cref{eq: gamma eps expansion} and $\ue \doteq \un + \e v$ away from the discontinuity, we obtain
\begin{equation*}
    \begin{split}
        &\ue^\pm(\t,t) \doteq \lim_{\delta \to 0} \ue \left( \g^0(\t,t) + \e \r(\t,t) n^0(\t,t) \pm \delta n^\e(\t,t),t\right) \\
        &\doteq \lim_{\delta \to 0} \Bigg( \Big( \un(x,t) + \e v(x,t) \Big)\Big \vert_{x=\g^0(\t,t) + \e \r(\t,t) n^0(\t,t) \pm \delta (n^0(\t,t)+ \e \eta(\t,t))}\Bigg).
    \end{split}
\end{equation*}
We continue using Taylor expansions
\begin{equation*}
    \begin{split}
        \ue^\pm(\t,t)  
        &\doteq \lim_{\delta \to 0}\Big ( \un(\g^0(\t,t) \pm \delta n^0(\t,t),t) + \e v(\g^0(\t,t) \pm \delta n^0(\t,t),t) \\ &\hspace{1.5cm}+ \e\left( \nabla_x \un(\g^0(\t,t)\pm\delta n^0(\t,t),t) \right) \cdot \left( \r(\t,t) n^0(\t,t) \pm \delta \eta(\t,t) \right) 
         \Big ) \\
        &= \un(\g^0(\t,t) \pm,t) + \e \left(\left( \nabla_x \un(\g^0(\t,t)\pm,t) \right) \cdot \left( \r(\t,t) n^0(\t,t)  \right) +  v(\g^0(\t,t)\pm,t) \right).
    \end{split}
\end{equation*}

Then, the difference of the linearizations reads
\begin{equation}\label{eq: expansion [u]}
    [\ue](\t,t) = \uebp(\t,t) - \uebm(\t,t) = \left[ \un \right](\t,t) + \e \left[ (\nabla_x \un) \cdot (\r n^0) + v \right](\t,t).
\end{equation}
Similarly, we compute the jump in the flux
\begin{equation*}
    \begin{split}
        \left[f(\ue)\right](\t,t) &= f(\uebp(\t,t))- f(\uebm(\t,t)) \\
        &\doteq \left[f(\un)\right](\t,t) + \e \left[f'(\un) \left( (\nabla_x \un) \cdot  \left(\r n^0\right) + v \right)\right](\t,t).
    \end{split}
\end{equation*}
Furthermore,
  \begin{equation}\label{eq: expansion f*n}
    \begin{split}
        &\bigl[f(\ue)\bigr](\t,t) \cdot n^\e(\t,t) \\
  &\doteq \Bigg(\Big[f\bigl(\un\bigr)\Big](\t,t) + \e  \Big[f'\bigl(\un\bigr)
     \bigl((\nabla_x \un)\cdot (\r
    n^0) + v\bigr)\Big](\t,t)\Bigg) \cdot \bigl(n^0(\t,t)
    + \e  {\eta}(\t,t)\bigr) \\
  &\doteq \Big[\!f\bigl(\un\bigr)\!\Big]\!(\t,t)\! \cdot \! n^0(\t,t) + \e \bigg(\!\Big[\!f'\bigl(\un\bigr)
     \bigl((\nabla_x \un) \!\cdot \!(\r
    n^0) \!+\! v\bigr)\!\Big]\!(\t,t)
    \!\cdot \!n^0(\t,t) + \Big[\!f\bigl(\un\bigr)\!\Big]\!(\t,t)
    \!\cdot \!{\eta}(\t,t)\!\bigg).
    \end{split}
  \end{equation}
  We collect all the terms to expand $s^\e$: 
  \begin{align}\label{eq:seps}
  s^\e(\t,t) \doteq \frac{\bigl[f(\ue)\bigr](\t,t) \cdot n^\e(\t,t)}{[\ue](\t,t)} 
  \doteq s^0(\t,t) + \e \lambda(\t,t) .
  \end{align}
Combining the expansions \cref{eq: expansion [u]} and \cref{eq: expansion f*n}, we obtain
\begin{equation}\label{eq: lambda}
    \begin{split}
        \lambda(\t,t) = &\frac{1}{[\un](\t,t)} \Bigg(\Big[f'\bigl(\un\bigr)
     \bigl((\nabla_x \un)\cdot (\r
    n^0 )+ v\bigr)\Big](\t,t)
    \cdot n^0(\t,t) + \Big[f\bigl(\un\bigr)\Big](\t,t)
    \cdot {\eta}(\t,t)\Bigg)\\
        &-\frac{1}{([\un](\t,t))^2} \Bigg( \left( \Big[f\bigl(\un\bigr)\Big](\t,t)\, \cdot n^0(\t,t) \right) \left[ (\nabla_x \un) \cdot (\r n^0) + v \right] (\t,t)
        \Bigg).
    \end{split}
\end{equation}
We now use this to derive the evolution equation for $\r(\t,t)$. We choose the evolution of the parametrization normal to the curve, such that by the Rankine-Hugoniot condition, it holds 
\begin{equation}\label{eq: RH}
    \partial_t \g^0(\t,t) = s^0(\t,t) n^0(\t,t), \quad
    \partial_t \g^\e(\t,t) = s^\e(\t,t) n^\e(\t,t).
\end{equation}
With the expansions \cref{eq: neps expansion} and \cref{eq:seps} we obtain 
\begin{equation}\label{eq: expansion time deriv gamma eps}
    \begin{split}
        \partial_t \g^\e(\t,t) &\doteq \left( s^0\left( \t,t \right) + \e \lambda(\t,t)\right) \left( n^0\left( \t,t \right) + \e \eta\left( \t,t \right) \right) \\ 
        &\doteq s^0(\t,t) n^0(\t,t) + \e \left( \lambda(\t,t) n^0(\t,t) + s^0(\t,t) \eta(\t,t) \right).
    \end{split}
\end{equation}
From the expansion of $\g^\e$ in \cref{eq: gamma eps expansion}, we derive 
\begin{equation*}
    \r(\t,t) \doteq \frac{1}{\e} \left( \g^\e(\t,t) - \g^0(\t,t) \right)\cdot  n^0(\t,t).
\end{equation*}
Differentiating this with respect to time yields
\begin{align*}
    \partial_t \r(\t,t) \doteq \frac{1}{\e} \left( \partial_t \g^\e(\t,t) - \partial_t \g^0(\t,t) \right)\cdot n^0(\t,t) + \frac{1}{\e} (\g^\e(\t,t)-\g^0(\t,t))\cdot \left( \partial_t n^0(\t,t) \right).
\end{align*}
Using the expansions in \cref{eq: gamma eps expansion} and \cref{eq: expansion time deriv gamma eps} and the Rankine-Hugoniot condition \cref{eq: RH}, we write 
    \begin{equation} 
        \begin{split}
            \partial_t \r(\t,t) &\doteq \frac{1}{\e} \Big( s^0(\t,t) n^0(\t,t) + \e \left( \lambda(\t,t) n^0(\t,t) + s^0(\t,t) \eta(\t,t) \right) - s^0(\t,t) n^0(\t,t) \Big)\cdot n^0(\t,t) \\
     &\hspace{1cm} + \frac{1}{\e} \Big(\g^0(\t,t) + \e \r(\t,t) n^0(\t,t)-\g^0(\t,t)\Big) \cdot \left( \partial_t n^0(\t,t) \right) \\ 
     &= \left( \lambda(\t,t) n^0(\t,t) + s^0(\t,t) \eta(\t,t) \right) \cdot n^0(\t,t) + \r(\t,t) n^0(\t,t) \cdot \left( \partial_t n^0(\t,t) \right)\\ 
     &=  \lambda(\t,t),
        \end{split}
    \end{equation}
where we use $0=  \partial_t \| n^0(\t,t) \|^2 = 2\left( n^0(\t,t) \cdot \left( \partial_t n^0(\t,t) \right) \right)$ and $\eta(\t,t) \cdot n^0(\t,t)=0$ in the last equation.
Substituting the formula for $\lambda$ from \cref{eq: lambda} we arrive at the evolution equation for $\r$.
Rearranging the terms  and inserting $\eta$ defined in \cref{eq: def eta} yields \cref{eq:evolution r}.

\end{proof}
Together with \cref{prop: well defined time}, we obtain from \cref{prop: regularity sol} the well-definedness of the regular variation as defined in \cref{def:reg vari} locally in time. We discuss shortly where locality in time is necessary.
\begin{remark}[Locality in time]
    For the reader's convenience, we summarize the assumptions that restrict the analysis to a local time interval. First of all, the assumption that no discontinuity other than the one at $\operatorname{Im}(\g^0)$ appears in \cref{assumptions} (a) is an assumption that can only hold locally in time, since discontinuities can arise even for smooth data in nonlinear balance laws, see e.g. \cite[Section 3.3]{lV}. Second, also \cref{prop: gamma c2,prop: well defined time,prop: regularity sol} for the well-definedness of $\t^*$ and therefore the well-definedness of the regular variation as defined in \cref{def:reg vari} only hold locally in time. 
    All of these statements pose restrictions on the time span $[0,T)$ we consider.
\end{remark}

As mentioned in \cref{sec: Intro}, Lecaros and Zuazua \cite{LZ16} developed a calculus for two-dimensional conservation laws. The relation to the work by Lecaros and Zuazua is clarified in the following Proposition

\begin{proposition}\label{prop: lec_zuazua} 
    For $d=2$, $S=0$ and discontinuity curves that divide the domain into two subdomains, the evolution equations for the tangent vectors in \cref{theorem} coincide with those in \cite{LZ16}.
\end{proposition}
For the proof we refer to the Appendix \ref{sec: equivalence}.

\FloatBarrier
\section{Implementation}\label{sec: Implementation}

The formal derivation in \cref{sec: gen tangent vec} provides a basis for the derivation of the numerical scheme.
In \cref{theorem} we introduced regular variations and evolution equations to compute the generalized tangent vector to $\ue$.
In the following, we want to validate the presented calculus using numerical approximations, where we restrict ourselves to the two-dimensional setting $d=2$. To fix the notation we discuss the reduction of the calculus in \cref{sec: 2d reduction}.
In \cref{sec: Methods} we introduce the methods to calculate the first-order approximation $\ueb$ following \cref{def:reg vari}. Afterwards, in \cref{sec:pde-solver} we give details on the numerical \textsc{PDE-solver} that we use.

Numerical experiments using these methods are shown in \cref{sec: num results} to provide numerical evidence that $\ueb$ is indeed a first-order approximation of $\ue$ in these examples.

\FloatBarrier
\subsection{Explicit Formulas in the case $d=2$}\label{sec: 2d reduction}
In two space dimensions, we provide explicit formulas for some quantities.
We start with the normal vector, that was defined in \cref{eq:definition normal} and is given by 
\begin{equation}\label{eq: normal 2d}
     n^0(\t,t)  \coloneqq  \frac{1}{ \| \partial_\t \g^0(\t,t) \|  }\rotmat \left( \partial_\t \g^0(\t,t) \right).
\end{equation}
The first-order term for the expansion of $n^\e$ is defined as
\begin{equation*}
    \eta(\t,t) = -\frac{\partial_\t \g^0(\t,t)}{\| \partial_\t \g^0(\t,t) \|^2}\left( \partial_\t r(\t,t) \right).
\end{equation*}
Using these in the proof of \cref{theorem} and rewriting this in the conservative form leads to \cref{cor:conservative r eq}. 

\FloatBarrier
\subsection{Methods}\label{sec: Methods}

Our overall goal is to compute the approximate solution $ \ueb(x,t) \doteq  \ue(x,t)$.
In this section, we present the numerical methods to obtain this approximation. We start by discussing how to combine all components for the numerical approximation of $\ueb(x,t)$ following \cref{eq: reg vari}. After this, we discuss the computation of the components needed in the first-order approximation, i.e., $\un$, $\g^0$, $v$, $\r$ and $\t^*$, using among others the evolution equations developed in \cref{theorem}.

In the following, we denote the numerical approximation by $\nue$. In Algorithm \ref{alg:general}, we present the procedure to approximate numerically the first-order approximation $\ueb$. At every time step, we choose the step size $\Delta t$ such that all CFL conditions of the partial differential equations \eqref{eq: IVP}, \eqref{eq:evolution v} and \eqref{eq:evolution r} for $\un$, $v$ and $\r$ are satisfied. Afterwards, we update $\gn$, $\nv$, $\nr$ and $\nun$, the numerical approximations of $\g^0$, $v$, $\r$ and $\un$ as explained in Sections \ref{sec: num sol}, \ref{sec:num shock curve}, \ref{sec: num sol v} and \ref{sec:num sol r}. Here, $f$, $\mathcal{F}_v$ and $\mathcal{F}_\r$ describe the fluxes for $\un$, $v$ and $\r$, respectively, while $S$, $\mathcal{S}_v$ and $\mathcal{S}_\r$ describe the source terms for $\un$, $v$ and $\r$, respectively. Finally, using the function \textsc{Approx}, we combine the components to obtain the approximation $\nue$ following \cref{eq: reg vari}.

In the function \textsc{Approx} in Algorithm \ref{alg:general}, we first calculate \textsc{theta\_star} as explained in \cref{sec: theta_stern} followed by the distance $\alpha$ that is measured relative to the displacement $\e \vert \nr\vert $. Afterwards we check whether the point $x$ lies in $M(t)$ from \cref{eq: reg vari} by checking if it is in an $\e \vert \nr \vert$ tube around $\gn$ by $\alpha \le 1$. By checking if 
\begin{equation*}
    \operatorname{dist}\left( \gn(\t^*)+ \e \alpha \nr(\t^*) \nn, x \right) < \operatorname{dist}\left( \gn(\t^*)- \e \alpha \nr(\t^*) \nn, x \right),
\end{equation*}
we ensure that it lies on the correct side of $\gn$. If this is true, we add the shock displacement term $-\operatorname{sign}(\nr(\t^*)) [\nun]$ otherwise, we only use $\nun + \e \nv$.

In the following sections we always assume a fixed time $t \ge 0$ and discuss how to compute the different quantities in a time-step of $\Delta t$.

\begin{algorithm}[H]
	\caption{Approximation of $\ueb$}\label{alg:general}

    \KwIn{initial data $\unn$, initial shock curve parametrization $\g_0^0$, initial $L^1$-displacement $v_0$, initial shock displacement $\r_0$, narrow band distance $\delta_{nb}$, smoothing distance $\delta^s$, jump evaluation distance $\delta_\pm$, final time $T$}
    \KwOut{$\nue(\cdot,T)$}
    $t \gets 0$\\
    $\left( \nun,\gn, \nv,\nr\right) \gets \textsc{Initialize}(\unn,\g_0^0,v_0,\r_0)$\\
    \While{$t < T$}{
        $\Delta t \gets \min\left( \textsc{CFL-time-step}(f),\textsc{CFL-time-step}(\mathcal{F}_v),\textsc{CFL-time-step}(\mathcal{F}_\r) \right)$\\
        $\gn \gets \textsc{time-step-shock-curve}(\gn,\Delta t,\nun,f(\nun),\delta_\pm)$\\
        
        $\nv \gets \textsc{time-step-$L^1$-displacement}(\mathcal{F}_v(v,x,t;\un),\nun,\gn,\nv,\Delta t,\delta_{nb},\delta^s)$\\
        $\nr \gets \textsc{PDE-solver}(\nr,\Delta t,\mathcal{F}_\r(\r,\t,t;\nun,\nv,\delta_\pm),\mathcal{S}_\r( \r,\t,t;\nun,\nv, \delta_\pm ))$\\
        $\nun \gets \textsc{PDE-solver}\big(\nun,\Delta t, f(\nun),S(\nun)\big)$\\
        $\nue \gets \textsc{Approx}\left( \nun, \gn,\nv,\nr,\delta_\pm\right)$\\
        $t \gets t + \Delta t$
    }
    \hspace*{-\algomargin}\rule{\dimexpr\linewidth\relax}{0.4pt}

\SetKwProg{Fn}{Function}{}{}
\Fn{$\textsc{Approx}\left( \nun, \gn,\nv,\nr,\delta_\pm\right)$}{

        $\t^* \gets \textsc{theta\_star}(x;\gn)$\\
        \If{$\vert \nr(\t^*)\vert > 0$}{
        $\alpha \gets \frac{\operatorname{dist}\left( \gn(\t^*),x \right)}{\e \vert \nr(\t^*) \vert}$\\
        $\operatorname{dist}_+ \gets \operatorname{dist}\left( \gn(\t^*)+ \e \alpha \nr(\t^*) \nn(\t^*), x \right)$ \\ 
        $\operatorname{dist}_- \gets \operatorname{dist}\left( \gn(\t^*)- \e \alpha \nr(\t^*) \nn(\t^*), x \right)$\\ 
        \If{$\alpha \le 1  \text{ and } \operatorname{dist}_+ < \operatorname{dist}_-$}{
                $[\nun] \gets \nun(\gn(\t^*) + \delta_\pm \nn(\t^*))-\nun(\gn(\t^*) - \delta_\pm \nn(\t^*)) $\\
                \KwRet $\nun(x) + \e \nv(x) - \operatorname{sign}(\nr(\t^*)) [\nun]$
            }
        }
    \KwRet $\nun(x) + \e \nv(x)$
}
\end{algorithm}

\FloatBarrier
\subsubsection{Computation of the Unperturbed Solution $\un$}\label{sec: num sol}

To obtain the approximation of the solution $\un$ to the initial value problem \cref{eq: IVP}, we apply a numerical scheme to solve hyperbolic balance laws. The numerical solution at a fixed time step is denoted by $\nun(\cdot): \Omega \to \R$. To handle the differential equations \eqref{eq:evolution v} and \eqref{eq:evolution r} as well, the scheme should be able to treat space- and time-dependent fluxes and source terms. Furthermore, for the evolution of the shock position, high accuracy and, ideally, a higher-order reconstruction are necessary. When using a piecewise constant representation together with the shock curve approximation in \cref{sec:num shock curve}, the strong jumps across every cell interface can lead to shock curve representations with incorrect normal vectors to the shock curve.

We refer to the time stepping function of the solver as $\textsc{PDE-solver}(\textsc{u},\Delta t, \textsc{f},\textsc{s})$, where $\textsc{u}$ describes the current state of the solution, $\Delta t$ the time step, $\textsc{f}$ the flux function and $\textsc{s}$ the source term. We give details on the solver in \cref{sec:pde-solver}.
\FloatBarrier
\subsubsection{Computation of the Shock Curve $\g^0$}\label{sec:num shock curve}

Recall that $C^2$ regularity of the curve is required for the well-definedness of $\t^*$ in \cref{prop: well defined} and the evolution equation \eqref{eq:evolution r} for $\r$, see \cref{prop: regularity sol}. Therefore, we approximate the 
parametrization $\g^0$ of the discontinuity curve by two-dimensional cubic splines, i.e., we have $N+1$ spline nodes $\sn^i\approx\g^0(\frac{i}{N})$ for $i=0,\dots,N$ and a set $C$ of coefficients for the polynomials in between the spline nodes for each dimension. Since we do not want to impose additional conditions on the boundary, we choose the not-a-knot closing condition \cite{Boor1985ConvergenceOC}.
We denote the spline function approximating $\g^0$ by 
\begin{equation*}
    \gn \in \{\g \in  C^2(\D,\R^d)  : \g\big\vert_{\left[\frac{i}{N},\frac{i+1}{N}\right]} \in (\Pi_3)^2, \, i=0,\dots,N-1 \},
\end{equation*}
where $(\Pi_3)^2$ describes the space of two-dimensional polynomials of degree at most three.
Following \cref{eq: normal 2d}, we compute the numerical approximation of the normal $n^0$ as
\begin{equation*}
    \nn(\t)  \coloneqq  \frac{1}{ \| \partial_\t \gn(\t) \|  }\rotmat \left( \partial_\t \gn(\t) \right).
\end{equation*}
For the time evolution of the shock curve, we use the Rankine-Hugoniot condition. For this purpose, we compute the shock speed at each spline node, i.e.,
\begin{equation}\label{eq: num deriv gamma}
    \begin{split}
        \partial_t \g^0(\t,t) &= \left( \frac{[f(u^0)](\t,t)}{[u^0](\t,t)} \cdot n^0(\t,t) \right) n^0(\t,t) \\
        &\approx \left( \frac{f(\nun(\sn^i + \delta_\pm \nn(\t)))-f(\nun(\sn^i - \delta_\pm \nn(\t)))}{\nun(\sn^i + \delta_\pm \nn(\t))-\nun(\sn^i - \delta_\pm \nn(\t))} \cdot \nn(\t) \right) \nn(\t),
    \end{split}
\end{equation}
for $\t=\frac{i}{N}$,
where the approximation of the right-hand side holds for fixed $t$.
The choice of the distance $\delta_\pm$ for the left- and right-hand states is governed by the interplay between the numerical viscosity of $\nun$ and the requirement that the states remain sufficiently close to the curve. Thus, it is strongly dependent on the \textsc{PDE-solver}.
The ordinary differential equations \eqref{eq: num deriv gamma} can then be solved by an arbitrary {\textsc{ODE-solver}$(\textsc{u},\Delta t, \dot{\textsc{u}})$} with \textsc{u} the state at the previous time step, $\Delta t$ the time step and $\dot{\textsc{u}}$ the time derivative of $\textsc{u}$. For simplicity, we use the explicit Euler method. 
After performing one time step, we obtain a set of new spline nodes, from which we calculate the spline coefficients for each segment with a suitable method. 
We refer to this time-stepping process as \textsc{time-step-shock-curve}.

\FloatBarrier
\subsubsection{Computation of the $L^1$-Displacement $v$}\label{sec: num sol v}

The $L^1$-infinitesimal displacement is computed using \cref{eq:evolution v},
where we define the flux as $\mathcal{F}_v(v,x,t;\un) \coloneqq f'(\un(x,t)) v$ and the source as $\mathcal{S}_v(v,x,t;\un) \coloneqq S'(\un(x,t)) v$.
Solving the equation for $v$ away from the discontinuity is straightforward and can be done in the same way as for $\un$, see \cref{sec: num sol}. Solving it near the discontinuity is more challenging. Here, the conservation of $v$ across the discontinuity of $\un$ is not necessarily satisfied. Therefore, simply computing the solution on the whole domain does not provide good results and may even lead to blow-ups. 
While one can use numerous strategies to avoid this problem, our approach is to correct the solution computed on the whole domain in a narrow band around the shock position.

For this purpose, we start from a modified solution $(\nun)^s$, which is smoothed within a narrow band around the shock position. The width of this band is determined by the distance parameter $\delta_{nb}>0$. It must be sufficiently large to ensure that the smoothing effect is not dominated by the inherent numerical viscosity, while an excessively large value leads to a loss of accuracy in the numerical solution for $v$. For simplicity, we use a box filter, see e.g. \cite[Section 3.2]{Sze22}.

Using $(\nun)^s$ in the transport equation \eqref{eq:evolution v} for $v$, we formally have a spatially discretized, continuous flux instead of a fully discontinuous one. Then, we perform a time step with the \textsc{PDE-solver} in \cref{sec: num sol}. Afterwards, the values in the narrow band with distance $\delta_{nb}+\delta^s$ are corrected by extrapolation. The additional distance $\delta^s>0$ is dependent on the applied smoothing technique. The narrow band with $\delta_{nb}+\delta^s$ should be sufficiently large to cover the areas in which the smoothing of $\nun$ is applied, but also small enough to keep as much of the original solution as possible. For the extrapolation, we choose values along the normal of the curve, but outside the narrow band.
For our purposes we found linear extrapolation along the normal to be sufficient. For the distance between the extrapolation points we choose $\Delta x$, which denotes the smallest diameter of a cell in the numerical scheme \textsc{PDE-solver}.
In Algorithm \ref{alg:general}, we write $\textsc{time-step-$L^1$-displacement}(\mathcal{F}_v(v,x,t;\un),\nun,\gn,\nv,\Delta t,\delta_{nb},\delta^s)$ for this time-stepping procedure for $\nv$.

\FloatBarrier
\subsubsection{Computation of the Shock Displacement $\r$}\label{sec:num sol r}

The transport equation for $\r$ is given by \cref{cor:conservative r eq}. In the following we refer to the flux as $\mathcal{F}_\r(\r,\t,t; \un,v, \delta_\pm) \coloneqq  G(\t,t) \r(\t,t)$ and $\mathcal{S}_r(\r,\t,t;\un,v,\delta_\pm) \coloneqq \tilde{G}(\t,t)  {\r(\t,t)} + \hat{g}(\t,t)$, where the dependence on $\un$, $v$ and $\delta_\pm$ is included in the functions $G$, $\tilde{G}$ and $\hat{g}$. Since $\g^0$ describes the shock position in $\un$, this also implies dependence on $\g^0$.
In particular, we need the left- and right-hand sided values of these functions. As already illustrated in \cref{sec:num shock curve}, we compute them using the parameter $\delta_\pm >0$, i.e.,
\begin{equation*}
    \left[\nun\right](\t)=\nun(\gn(\t) + \delta_\pm \nn(\t))-\nun(\gn(\t) - \delta_\pm \nn(\t))
\end{equation*}
and similarly for $\nv$, $f(\nun)$, and the corresponding combinations with $\nn$, following \cref{eq: jump definition}.
Finally, to solve \cref{eq:evolution r}, we use \textsc{PDE-solver} from \cref{sec: num sol}. We denote the numerical solution by $\nr$.
\begin{remark}
     Due to the discretizations of these quantities, the numerical solution $\nr$ to \cref{eq:evolution r} may exhibit oscillations. In experiments, we observed that adjusting the value $\delta_\pm$ and setting the finest grid sizes for $\nun$ and $\nv$ to $(\Delta x)^2$ and $\Delta x$ for $\nr$ while $\frac{1}{N}\gg \Delta x$ helps to reduce oscillations in $\nr$. 
\end{remark}
\FloatBarrier
\subsubsection{Computation of the Projection $\t^*$}\label{sec: theta_stern}
To compute $\t^*(y,t)$ as defined in \cref{eq: theta_stern}, we use a gradient descent method to minimize the functional 
\begin{equation*}
    \mathcal{J}(\t) \coloneqq  \| y- \g^0(\t,t) \|^2 \approx \| y- \gn(\t,t) \|^2,
\end{equation*}
for a fixed $t$, which has the same minimizers as $\|y-\gn(\t,t) \|$.
In general, $\mathcal{J}$ is not convex in $\t$, since $\gn$ is in general not linear in $\t$.
To find the minimizer numerically, we choose the starting points close to the argument $y$. Moreover, we perform each minimization from multiple initial points to reduce the risk of converging to local minima due to the non-convexity of the functional.

Starting from each of these points, we perform a fixed number of gradient-descent steps. Finally, we evaluate $\mathcal{J}$ at all candidates and compare the resulting values. We choose the candidate with the smallest value as $\t^*$. We refer to this procedure as \textsc{theta\_star}.

\FloatBarrier
\subsection{Numerical \textsc{PDE-solver}}\label{sec:pde-solver}
In Sections \ref{sec: num sol}, \ref{sec: num sol v} and \ref{sec:num sol r}, we use a \textsc{PDE-solver}.
In particular, these numerical simulations are performed using the strong-stability-preserving Runge-Kutta discontinuous Galerkin (SSP-RK-DG) \cite{Cockburn1990} framework \textit{MultiWave} \cite{kolb2026multiwavecomputationallabadaptive}, which features adaptive mesh refinement based on multiresolution analysis (MRA) \cite{Gerhard2016}. The latter is particularly well suited for the present application. Indeed, for sufficiently small perturbations, an accurate first-order approximation requires a mesh that is capable of resolving perturbations occurring on correspondingly small spatial scales. Employing a uniformly refined mesh would therefore result in high computational costs. In contrast, MRA-based grid adaptation dynamically refines the mesh only in regions where fine-scale structures or discontinuities are present, thereby achieving locally the required accuracy at significantly reduced computational cost.

For solving the equation for $\r$ as described in \cref{sec:num sol r}, we found it useful to steer the adaptivity taking also into account the curvature of $\gn$. The curvature of the spline can easily be evaluated as the combination of the derivatives in the spline segments, see \cref{sec:num shock curve}.

The numerical flux for the evolution equation for $v$ depends on evaluations of $\un$, whereas the one for the equation of $r$ depends on both $\un$ and $v$. Since the discontinuous Galerkin approximation is discontinuous across element interfaces, interface values are not uniquely defined. Whenever such values are required, we employ the arithmetic average of the values from all adjacent cells.

\FloatBarrier
\section{Computational Results}\label{sec: num results}
 
In this section, we present numerical experiments for the first-order approximation following the methods in \cref{sec: Implementation}. 
After describing how to measure the accuracy of the solutions in \cref{sec: acc}, we investigate two numerical examples in \cref{sec: halbkreis,sec: sinus}. The first example shows the case where the shock curve has endpoints in the domain and the second example displays the case where the domain is split in two parts by the curve and where a source term is included in the balance law.

We use a third-order DG-scheme with quadratic polynomials on Cartesian grids and an explicit third-order SSP–RK method with three stages for the time-discretization. For the numerical flux, we choose the local Lax–Friedrichs flux with the Shu limiter \cite{CS98}.
Moreover, we choose the smallest diameter of a cell to be $(\Delta x)^2$ for $\un$ and $v$ and  $\Delta x$ for $\r$, where $\Delta x$ depends on the example. The CFL numbers for the applications of \textsc{PDE-solver} are set to $0.5$ for both examples. 

\subsection{Accuracy}\label{sec: acc}
To assess the numerical approximation of the first-order approximation, we compare the corresponding first-order approximation computed by Algorithm \ref{alg:general} with the numerical solution obtained from perturbed initial data, denoted by $\nueob$.
Since the computations are performed on adaptive grids, we project both solutions onto a common fully refined grid before computing their $L^1$-distance.
To check for linear decay, we compute the weighted $L^1$-distance:
\begin{equation}\label{eq: weighted L1 error}
    e (\e,T)=\frac{\| \nue(\cdot,T) - \nueob(\cdot, T) \|_{L^1(\Omega)}  }{\e}.
\end{equation}
We expect an approximately linear decay, i.e. $e(\e,T)= \mathcal{O}(\e)$, which corresponds to an $L^1$-error of order $\mathcal{O}(\e^2)$, and in particular $o(\e)$.
This yields numerical evidence that the proposed tangent vector yields a first-order variation.
To quantify how well a linear model describes this decay, we calculate an affine least-squares fit and compute the mean error by 
\begin{equation}
    \operatorname{ME}  \coloneqq  \frac{1}{n} \sum_{i=1}^n \vert y_i - \hat{y}_i \vert,
\end{equation}
for $n$ the number of points, $\hat{y}_i$ the value of the affine least-squares fit in the $i$-th data point and $y_i$ the $i$-th data point.

\FloatBarrier
\subsection{Shock Curve with Endpoints in the Domain}\label{sec: halbkreis}

We consider an example in which the shock curve has both endpoints in the interior of the domain.
We use the two-dimensional Burgers' equation with flux $f=\left( \frac{1}{2} u^2,\frac{1}{2} u^2 \right)^T$ in \cref{eq: IVP} on the domain $[-0.3,0.7]^2$. The unperturbed initial data are given by 
\begin{equation}\label{eq: ini data halbkr u0}
    \unn(x) = \begin{cases}
        \frac{(x_1-x_2)+1- 2\big(2(x_1-x_2)^2-1\big)^3}{5} & \text{for }  \| x \|<\frac{1}{2} \text{ or } x_1+x_2\le 0,\\ 
        \frac{(x_1-x_2)+1}{5} & \text{otherwise},
    \end{cases}
\end{equation}
while the perturbed initial data are given by 
\begin{equation}\label{eq: ini data halbkr ue}
    \uen(x) = (1+2\e) \unn(x).
\end{equation}
In this example, both $\unn$ and $\uen$ are $C^2$ away from the discontinuity curve. 
With perturbation, the jump values are higher, which leads to faster shock propagation. Moreover, the jump values vary along the curve, causing the geometry of the shock curve to change as well. 

The initial shock curve is parametrized by  $\g_0^0(\t)= \Big( \frac{1}{2} \cos\left( (\t- \frac{1}{4}) \pi \right), \frac{1}{2} \sin\left( (\t- \frac{1}{4}) \pi \right) \Big)^T$ and the initial tangent vector is given by $(v_0,r_0)$ with
\begin{equation*}
    \begin{split}
        v_0(x) &= \begin{cases}
        2\frac{(x_1-x_2)+1- 2\big(2(x_1-x_2)^2-1\big)^3}{5} & \text{for }  \| x \|<\frac{1}{2} \text{ or } x_1+x_2\le 0,\\ 
        2\frac{(x_1-x_2)+1}{5} & \text{otherwise},
    \end{cases} \\ \text{ and } r_0 &\equiv 0.
    \end{split}
\end{equation*}
We choose the final time $T=0.2$, the smallest diameter of the grid as $\Delta x=\frac{1}{1280} $ and the number of spline nodes as $N=51$.

\cref{fig: halbkreis evo} shows the time evolution of the unperturbed solution $\nun$, the perturbed solution $\nueob$ and the first-order approximation $\nue$ computed by Algorithm \ref{alg:general} for $t= 0.1, 0.2$ for $\e =0.3$. We see that $\nun$ and $\nueob$ differ not only in their function values in the smooth regions but also in the position of the shock curve. 
The spatially varying jump values lead to a visibly different deformation of the perturbed shock curve in $\nueob$.
Comparing $\nueob$ and $\nue$, we observe differences in the solution values in the region between the shock curve of $\nun$ and the first-order approximation of the perturbed shock curve. Moreover, we clearly observe changes in the solution values in this region along the curve.

\cref{fig: halbkreis error}(a) shows the pointwise difference $\nueob - 
\nue$. Here, we see that the largest discrepancy occurs between the shock curve of the unperturbed solution $\nun$ and that of the perturbed solution $\nueob$.
In \cref{fig: halbkreis error}(b) we show the weighted $L^1$-error defined in \cref{eq: weighted L1 error} for different values of $\e$ at the final time $T=0.2$. For better illustration, we additionally show an affine least-squares fit to the data with a ME of approximately $6.01796 \times 10^{-5}$. Compared to the values of the plotted errors and the grid size $\Delta x$, the mean error is sufficiently small. Thus, the expected linear decay of $e(\e,0.2)$ is observed in this example. This provides numerical evidence that the proposed tangent vector correctly captures the first-order variation of $\ue$.

\captionsetup[subfigure]{justification=raggedright,singlelinecheck=false}

\begin{figure}[H]
\centering

\begin{subfigure}[t]{0.47\textwidth}
    \includegraphics[width=0.83\linewidth]{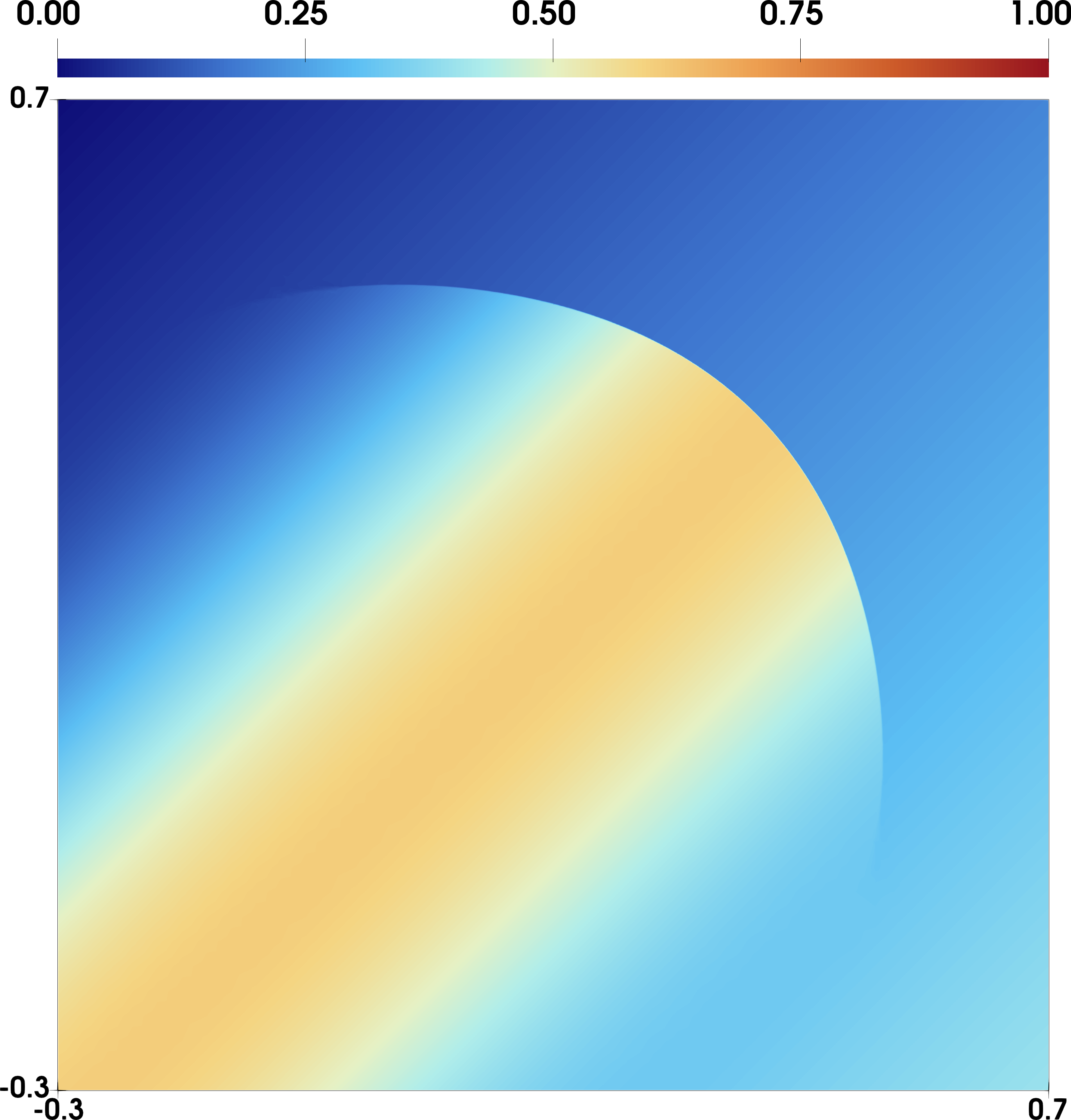}
    \caption{Unperturbed solution $\nun$ at $t=0.1$}
\end{subfigure}
\hfill
\begin{subfigure}[t]{0.47\textwidth}
    \includegraphics[width=0.83\linewidth]{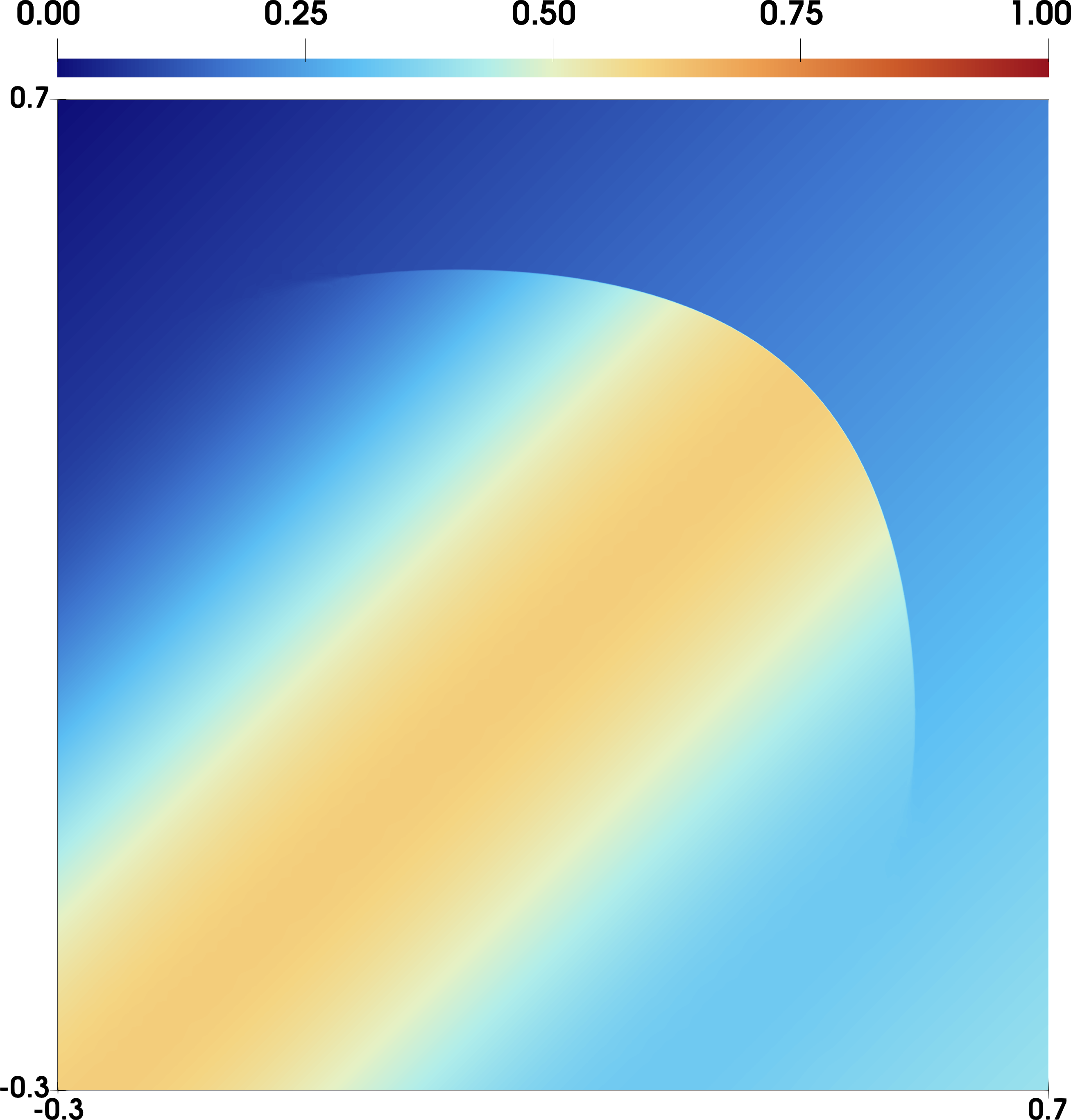}
    \caption{Unperturbed solution $\nun$ at $T=0.2$}
\end{subfigure}

\begin{subfigure}[t]{0.47\textwidth}
    \includegraphics[width=0.83\linewidth]{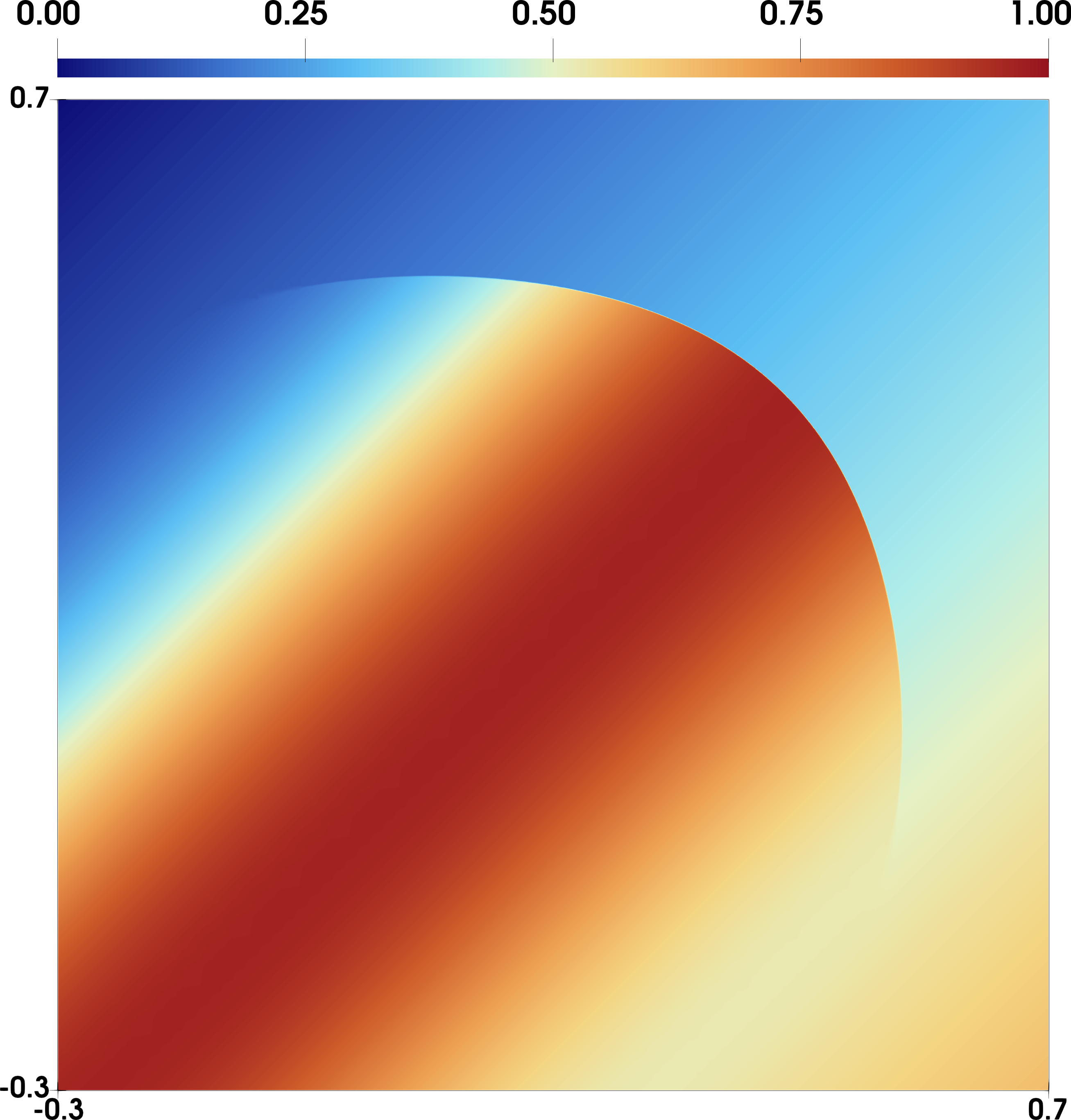}
    \caption{Perturbed solution $\nueob$ at $t=0.1$}
\end{subfigure}
\hfill
\begin{subfigure}[t]{0.47\textwidth}
    \includegraphics[width=0.83\linewidth]{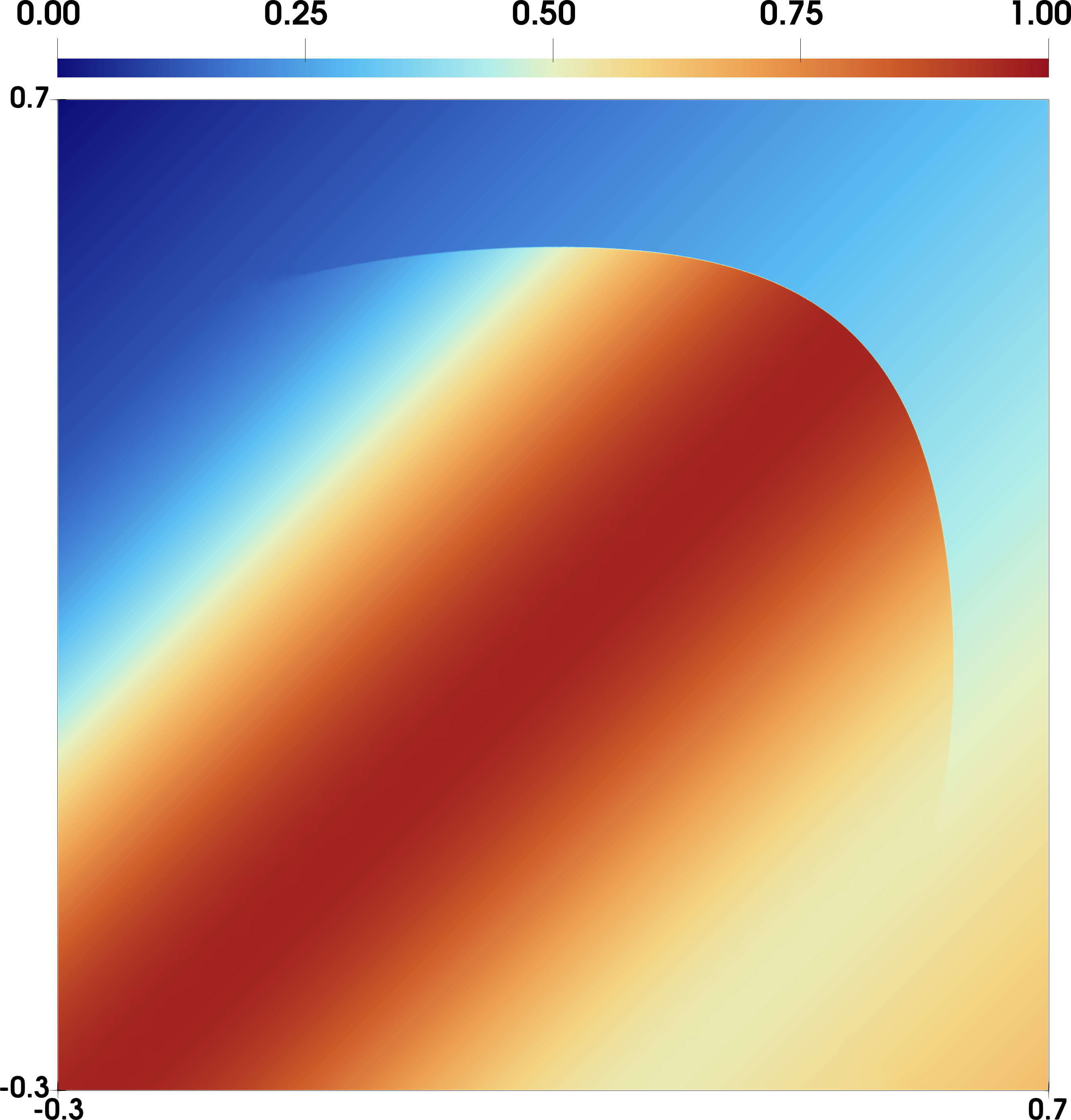}
    \caption{Perturbed solution $\nueob$ at $T=0.2$}
\end{subfigure}

\begin{subfigure}[t]{0.47\textwidth}
    \includegraphics[width=0.83\linewidth]{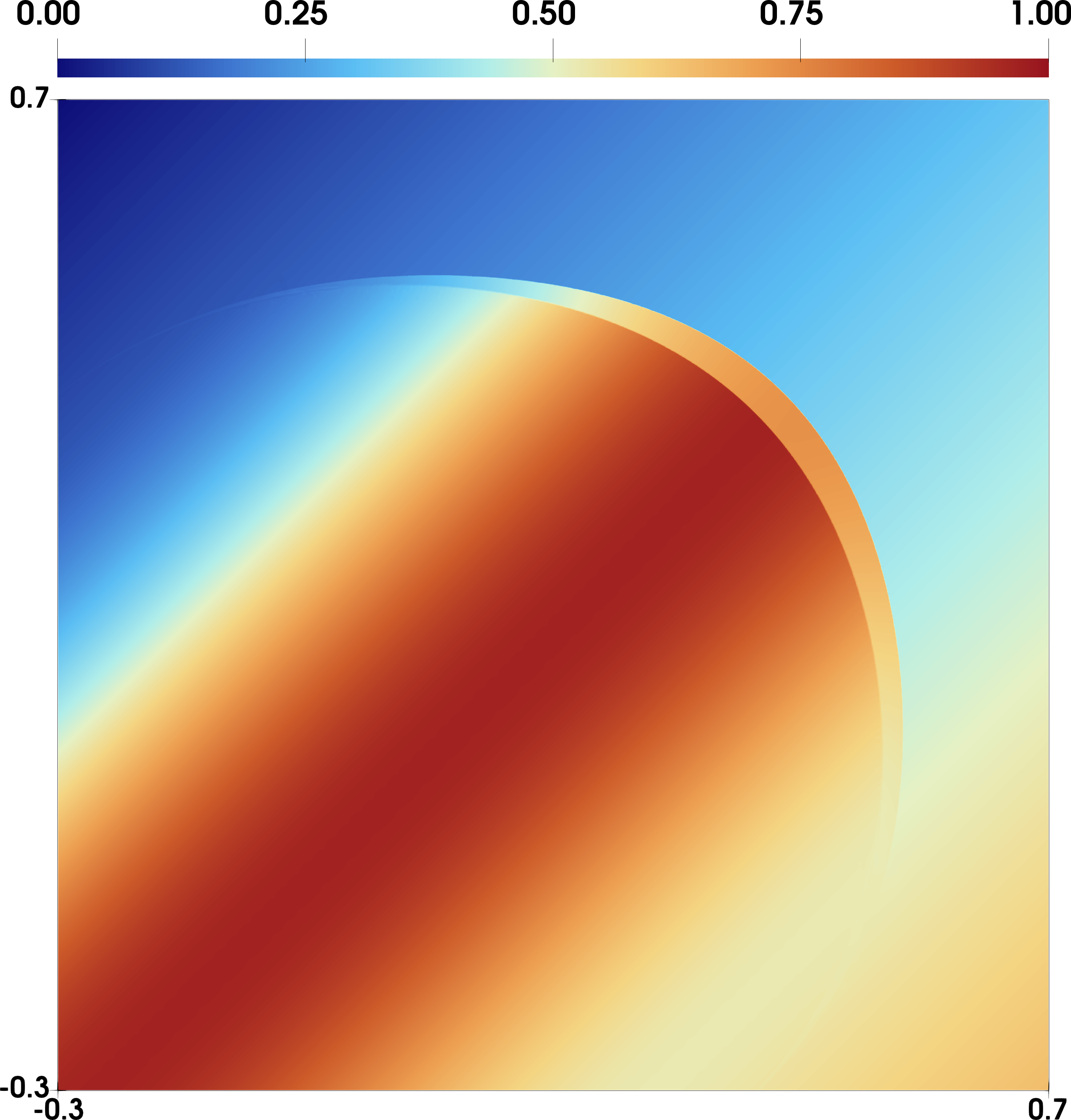}
    \caption{First-order approximation $\nue$ at $t=0.1$}
\end{subfigure}
\hfill
\begin{subfigure}[t]{0.47\textwidth}
    \includegraphics[width=0.83\linewidth]{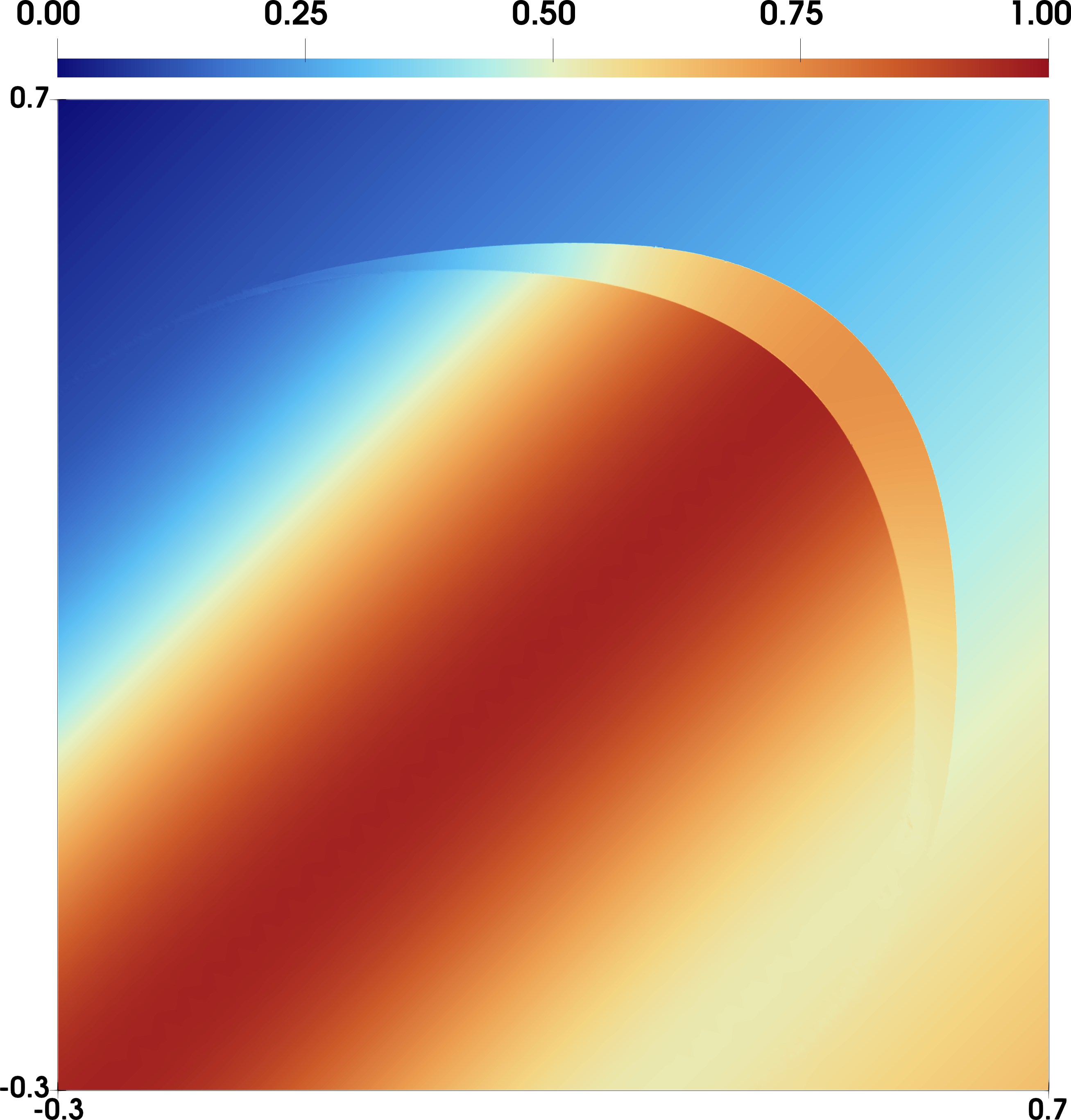}
    \caption{First-order approximation $\nue$ at $T=0.2$}
\end{subfigure}
\caption{Comparison of the numerical approximations of the unperturbed solution $\nun$ to initial data \cref{eq: ini data halbkr u0}, the perturbed solution $\nueob$ to initial data \cref{eq: ini data halbkr ue}, and the first-order approximation $\nue$ computed using Algorithm \ref{alg:general} for $\e=0.3$ at $t = 0.1$ and $T=0.2$.}
\label{fig: halbkreis evo}
\end{figure}
\captionsetup[subfigure]{justification=centering}

\begin{figure}
\centering
\begin{subfigure}[t]{0.44\textwidth}
    \includegraphics[width=\linewidth]{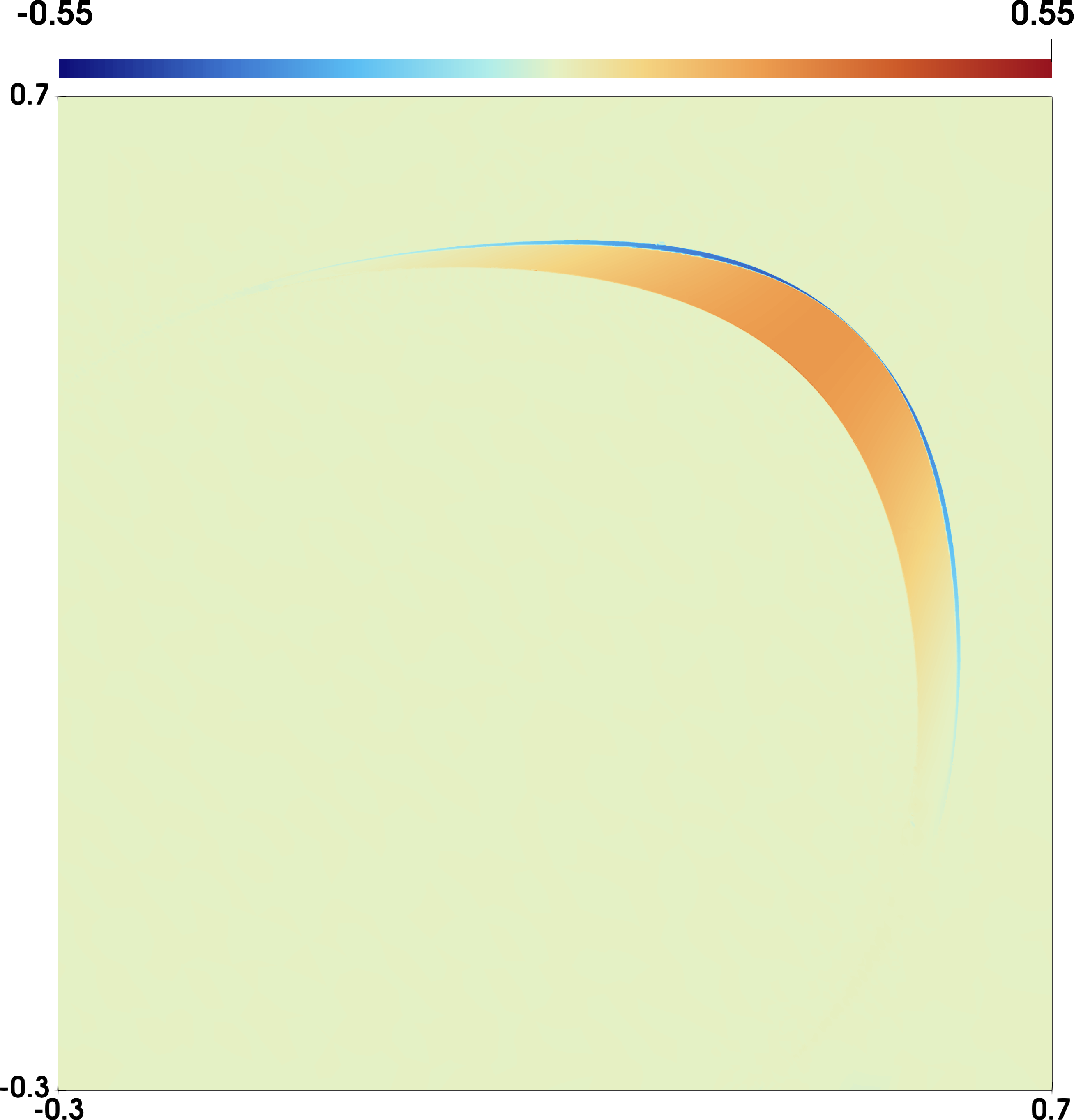}
    \caption{Pointwise difference $\nueob - \nue$}
\end{subfigure}
\hfill
\begin{subfigure}[t]{0.55\textwidth}
    \includegraphics[width=\linewidth]{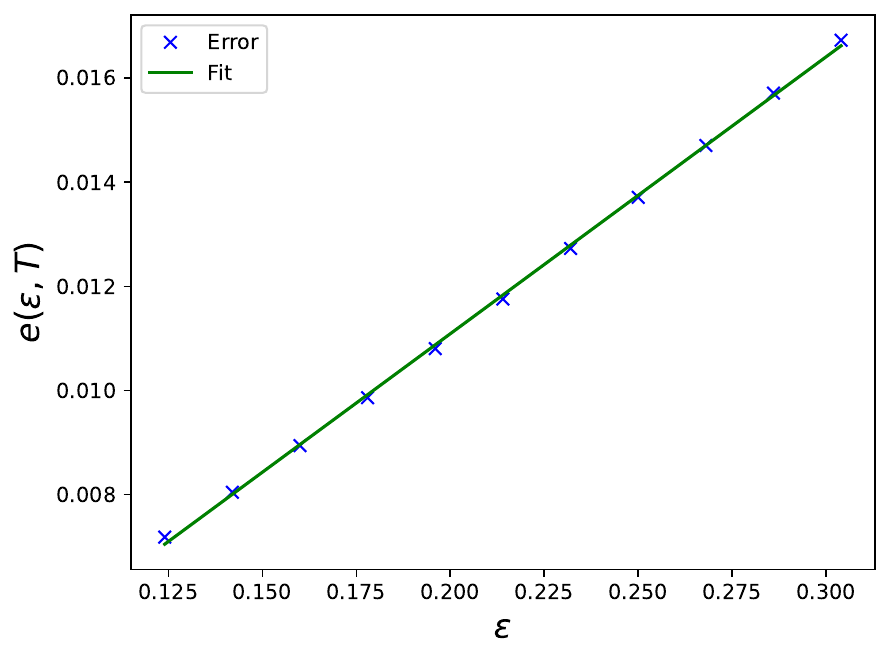}
\caption{Weighted $L^1$-distance $e(\cdot,T)$ between $\nueob$ and $\nue$}
\end{subfigure}

\caption{Difference between the numerical perturbed solution $\nueob$ and the numerical first-order approximation $\nue$ and the weighted $L^1$-distance between them at time $T=0.2$.}
\label{fig: halbkreis error}
\end{figure}

\FloatBarrier
\subsection{Shock Curve Separating the Domain}\label{sec: sinus}
This example illustrates a setting in which the discontinuity curve separates the domain into two connected components.
We consider the balance law with flux $f=\left( \frac{1}{3} (u+2)^3,\frac{1}{4} (u+2)^4 \right)^T$ and source $S(u)=10 u\left( u-1 \right)\left( u-2 \right)$ for \cref{eq: IVP}, on the domain $[-1,1] \times [-0.5,0.5]$. The unperturbed initial data are given by 
\begin{equation}\label{eq: ini data sin u0}
    \unn(x) = \begin{cases}
        \frac{3-2x_1^2}{5} & \text{for } x_2 \le \frac{1}{4}\sin(\pi\left( x_1 +1\right)),\\
        -\frac{1}{5} & \text{otherwise} ,
    \end{cases}
\end{equation}
while the perturbed initial data are given by 
\begin{equation}\label{eq: ini data sin ue}
    \uen(x) = \begin{cases}
        \frac{3+ 4\e -\left( 2+\e \right)x_1^2}{5} & \text{for } x_2 \le \frac{1}{4}\sin(\pi\left( x_1 +1\right)) + \frac{\e}{20} \sqrt{1 + \frac{1}{16} \pi^2  \cos(\pi(x_1+1))^2 },\\
        -\frac{1}{5} & \text{otherwise}.
    \end{cases}
\end{equation}
In this example, the initial discontinuity curve is perturbed as well.
The initial shock curve is parametrized by  $\g_0^0(\t)= \left( -1 + 2\t, \frac{1}{4} \sin(2\t \pi)  \right)^T$ and the initial tangent vector is given by $(v_0,r_0)$ with
\begin{equation*}
    v_0(x) = \begin{cases}
        \frac{4-x_1^2}{5} & \text{for } x_2 \le \frac{1}{4}\sin(\pi\left( x_1 +1\right)),\\
        0 & \text{otherwise},
    \end{cases} \qquad r_0 \equiv \frac{1}{20}.
\end{equation*}
We use $T=0.01$, the smallest diameter of the grid $\Delta x= \frac{1}{640} $ and the number of spline nodes $N=60$. \\
\cref{fig: sinus evo} shows the time evolution of the unperturbed solution $\nun$, the perturbed solution $\nueob$ and the first-order approximation $\nue$ computed by Algorithm \ref{alg:general} with $\e =0.3$ for $t=0.005, 0.01$. In this example, the discontinuity curve is already perturbed at the initial time $t=0$. 
Differences in both the shock curves and the solution values are also visible when comparing $\nun$ and $\nueob$.
The geometry of the shock curve changes as well, due to the variation of the jump values along the curve.

\captionsetup[subfigure]{justification=raggedright,singlelinecheck=false}
\begin{figure}[H]
\centering
\begin{subfigure}[t]{0.49\textwidth}
    \includegraphics[width=\linewidth]{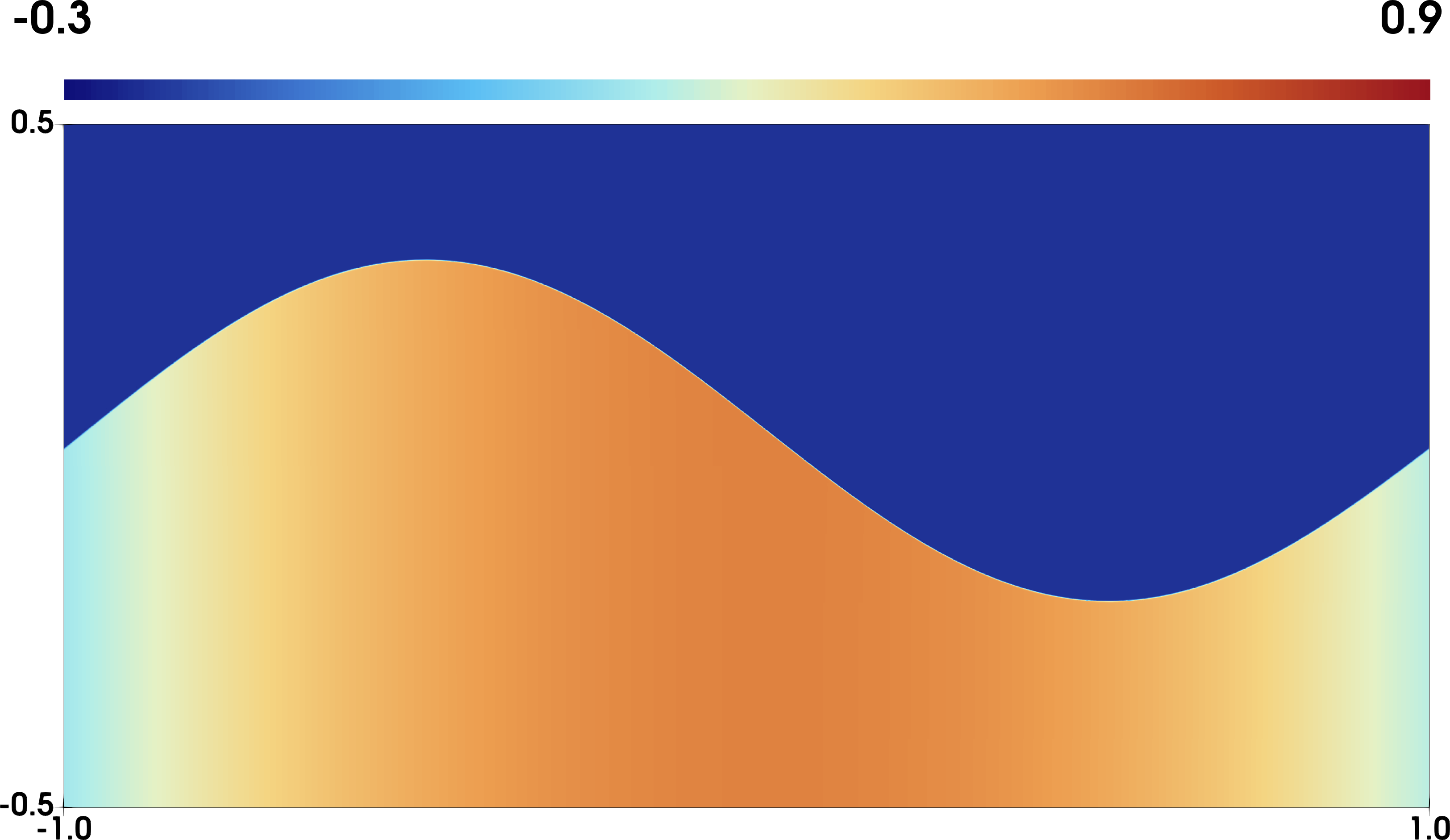}
    \caption{Unperturbed solution $\nun$ at $t=0.005$}
\end{subfigure}
\hfill
\begin{subfigure}[t]{0.49\textwidth}
    \includegraphics[width=\linewidth]{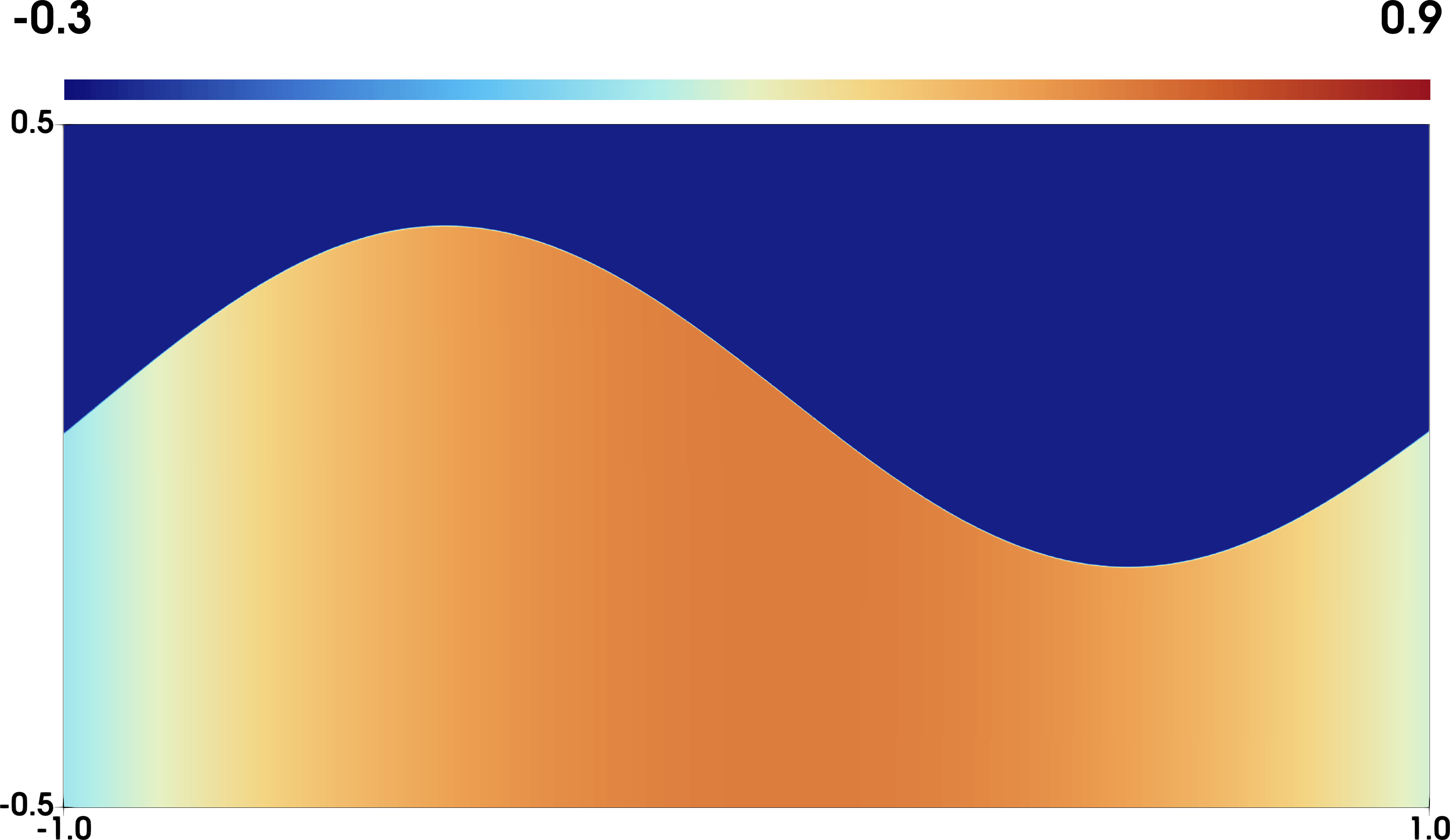}
    \caption{Unperturbed solution $\nun$ at $T=0.01$}
\end{subfigure} 

\begin{subfigure}[t]{0.49\textwidth}
    \includegraphics[width=\linewidth]{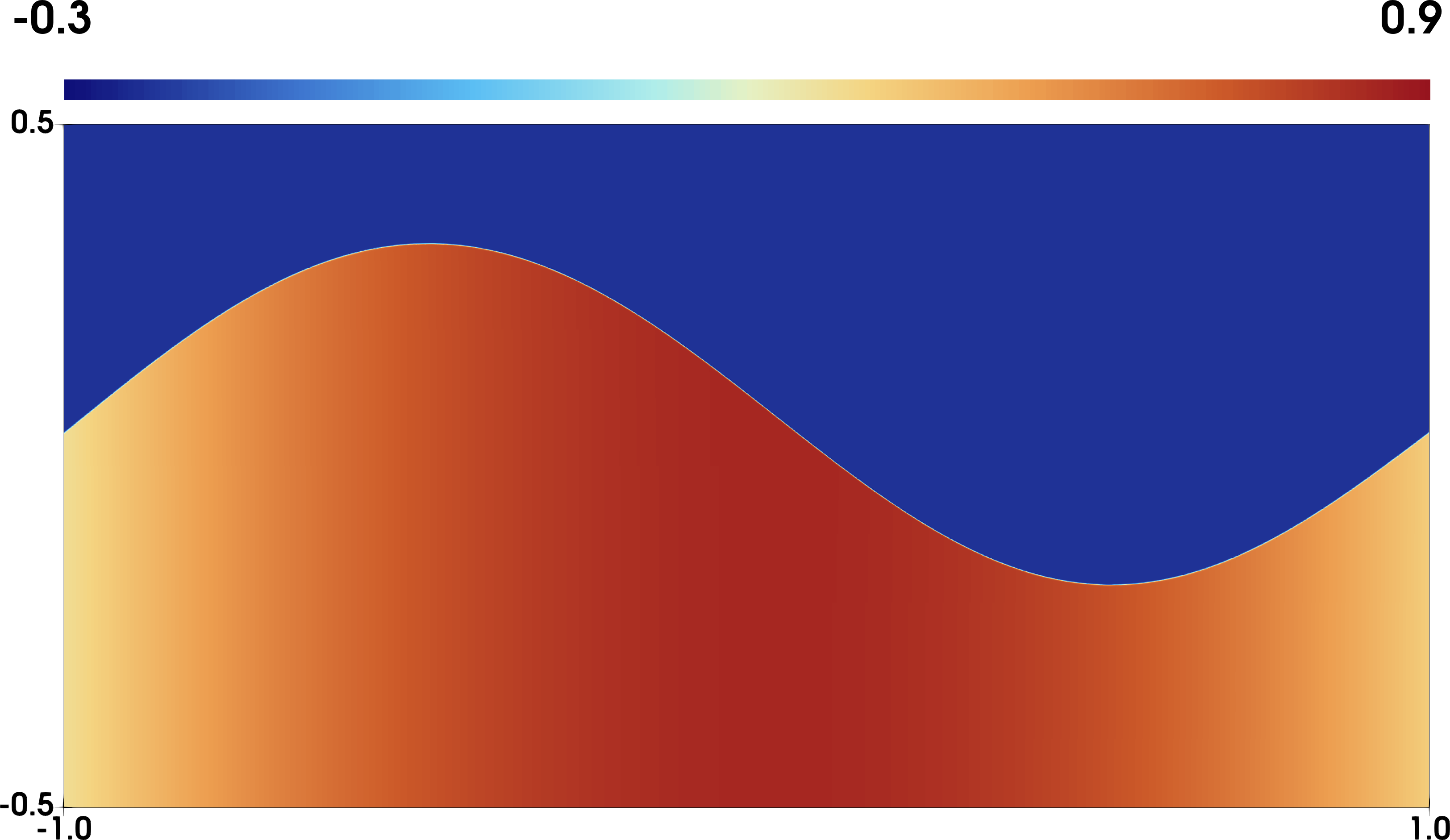}
    \caption{Perturbed solution $\nueob$ at $t=0.005$}
\end{subfigure}
\hfill
\begin{subfigure}[t]{0.49\textwidth}
    \includegraphics[width=\linewidth]{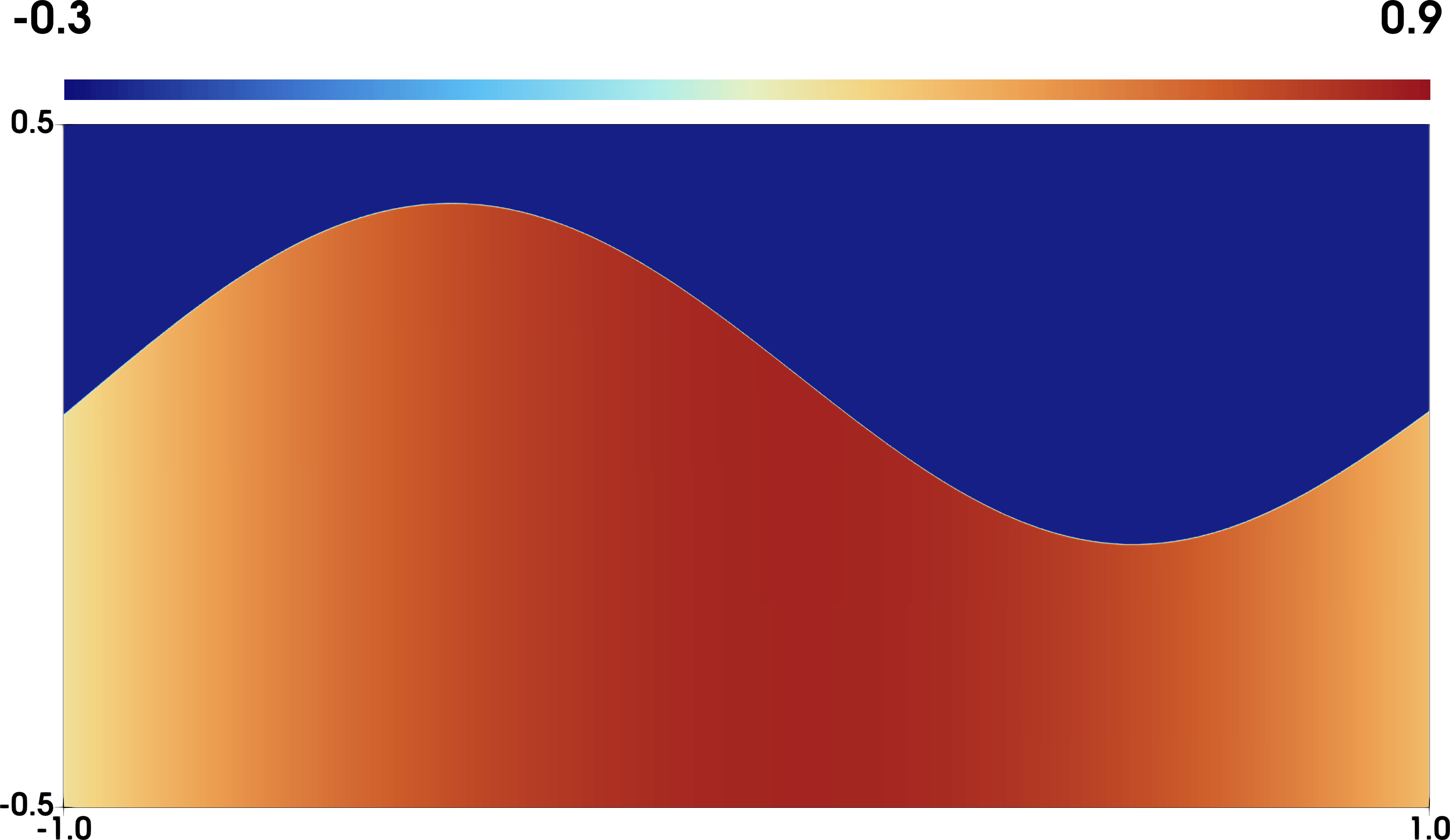}
    \caption{Perturbed solution $\nueob$ at $T=0.01$}
\end{subfigure} 

\begin{subfigure}[t]{0.49\textwidth}
    \includegraphics[width=\linewidth]{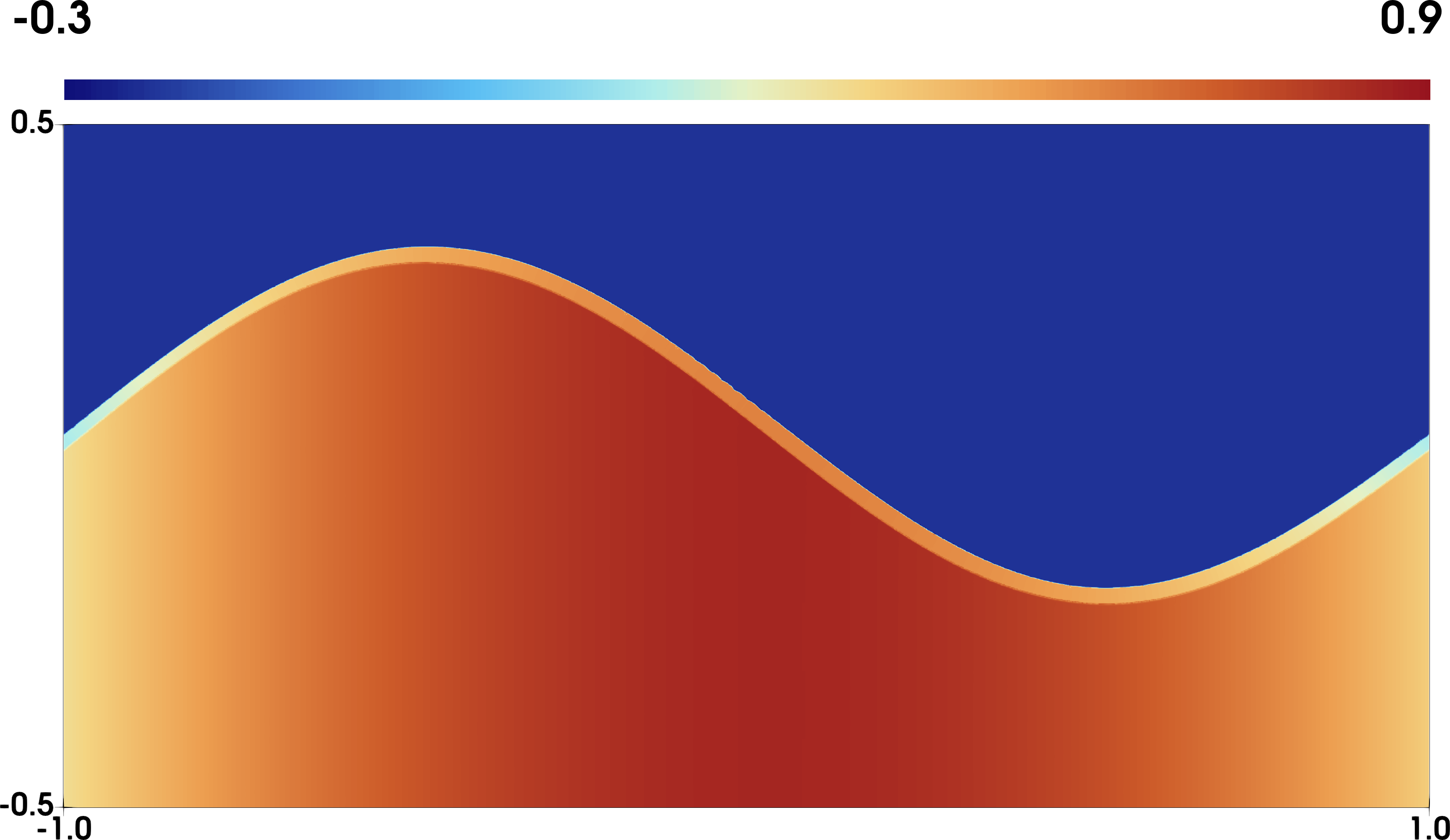}
    \caption{First-order approximation $\nue$ at $t=0.005$}
\end{subfigure}
\hfill
\begin{subfigure}[t]{0.49\textwidth}
    \includegraphics[width=\linewidth]{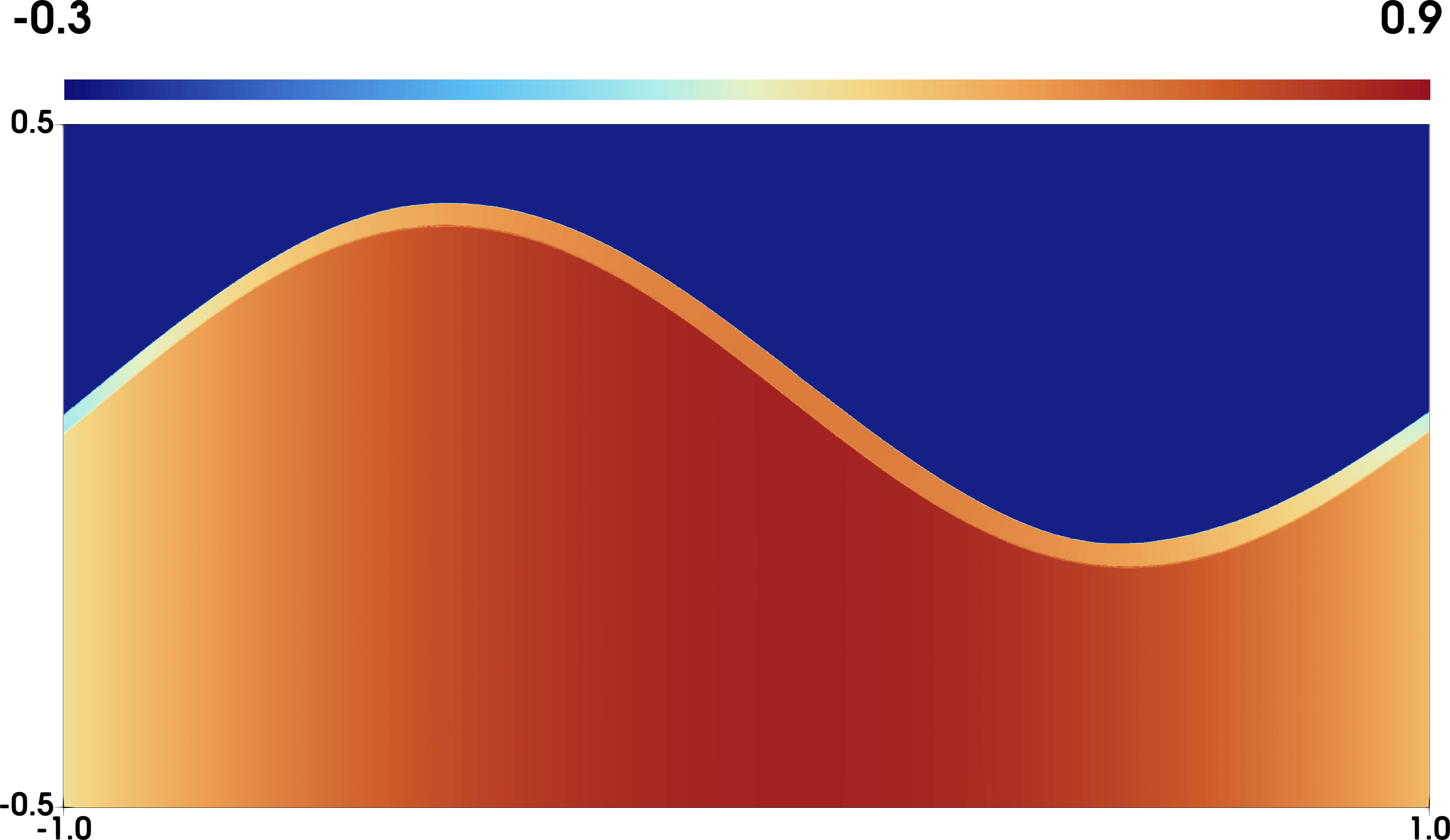}
    \caption{First-order approximation $\nue$ at $T=0.01$}
\end{subfigure}
\caption{Comparison of the numerical approximations of the unperturbed solution $\nun$ to initial data \cref{eq: ini data sin u0}, the perturbed solution $\nueob$ to initial data \cref{eq: ini data sin ue}, and the first-order approximation $\nue$ computed using Algorithm \ref{alg:general} for $\e=0.3$ at $t = 0.005$ and $T=0.01$.}
\label{fig: sinus evo}
\end{figure}
\captionsetup[subfigure]{justification=centering}

\cref{fig: sinus error}(a) shows the difference $\nueob - \nue$. Here, we see that the largest discrepancy lies between the shock curve of the unperturbed solution $\nun$ and that of the perturbed solution $\nueob$.
In \cref{fig: sinus error}(b) we show the weighted $L^1$-error following \cref{eq: weighted L1 error} for different values of $\e$ at final time $T=0.01$. For better illustration, we additionally show an affine least-squares fit to the data with a ME of approximately $2.66988\times 10^{-4}$. 
Compared to the values of the plotted errors and the grid size $\Delta x$, the mean error is sufficiently small. 
The observed behavior of $e(\e,0.01)$ is again consistent with the expected first-order behavior.

~

\begin{figure}[H]
\centering
\begin{subfigure}[t]{0.44\textwidth}
    \raisebox{0.7cm}{\includegraphics[width=\linewidth]{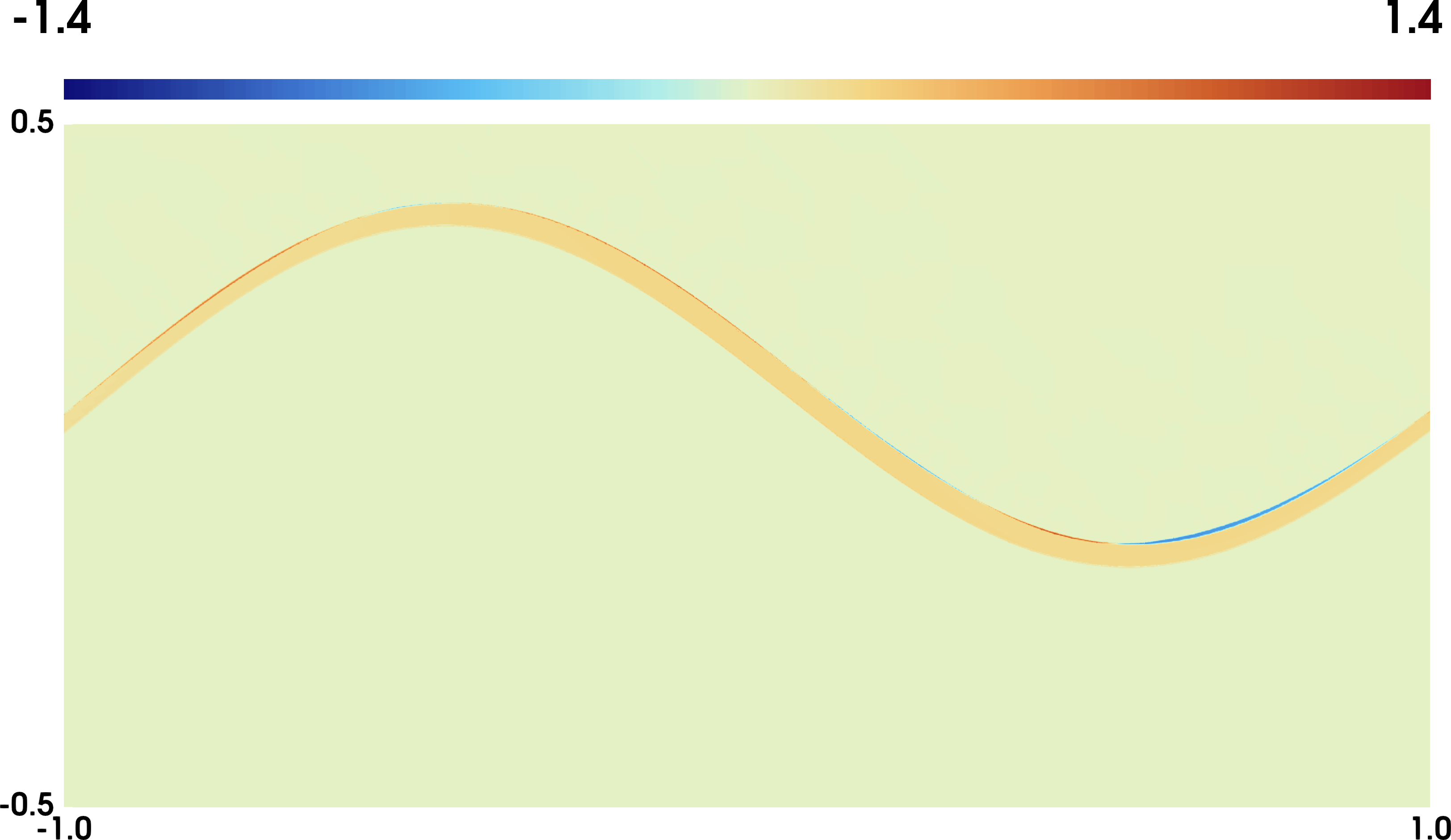}}
    \caption{Pointwise difference $\nueob - \nue$}
\end{subfigure}
\begin{subfigure}[t]{0.55\textwidth}
    \includegraphics[width=\linewidth]{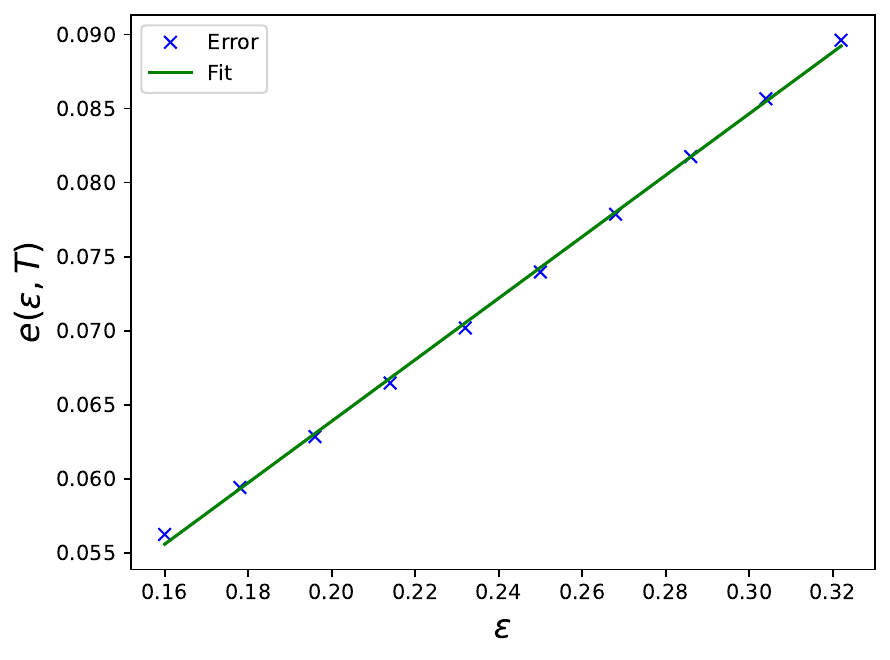}
\caption{Weighted $L^1$-distance $e(\cdot,T)$ between $\nueob$ and $\nue$}
\end{subfigure}
\caption{Difference between the numerical perturbed solution $\nueob$ and the numerical first-order approximation $\nue$ and the weighted $L^1$-distance between them at time $T=0.01$.}
\label{fig: sinus error}
\end{figure}

\FloatBarrier
\section{Conclusion}\label{sec: conclusion}

In this work, we have presented a first-order variational calculus for perturbations of initial data for multi-dimensional balance laws with discontinuities. We have studied the well-definedness of the resulting first-order approximation and have derived evolution equations for the corresponding tangent vector. In addition, we have discussed numerical methods for computing the first-order approximation.

The numerical experiments have shown that the proposed calculus yields a first-order variation for different configurations of the shock hypersurface. In particular, the observed behavior of the normalized $L^1$-error is consistent with the expected convergence as the perturbation parameter tends to zero.

In future work, we aim to prove the existence of tangent vectors and a regular variation of $\ue$.
Furthermore, the numerical representation of the shock curve could be improved by determining its position directly from the adaptive grid.
Although this could improve the accuracy, especially the reparametrization of the shock curve may become challenging. 
To overcome this, we want to represent the shock curves by a level set formulation \cite{Set99}.

\section*{Acknowledgements}
This work was funded by the Deutsche Forschungsgemeinschaft (DFG, German Research Foundation) – 320021702/GRK2326 – Energy, Entropy and Dissipative Dynamics (EDDy) and SPP 2410 (Hyperbolic Balance Laws in Fluid Mechanics: Complexity, Scales, Randomness) within the project 525842915.\\

\bibliographystyle{abbrvurl} 
\bibliography{references.bib}

\appendix

\section{Appendix - Normal Vectors in $\R^d$}\label{sec: normals Rd}
In this section we provide some background information on the normal to the shock curve and its expansion.
From standard theory of hypersurfaces, see e.g. \cite[Section 1.5]{defnormal}, we define the normal
\begin{equation*}
    n^\e:= \frac{\widehat{n^\e}}{ \| \widehat{n^\e} \|  }:=\frac{\bigtimes\left( \partial_{\t_1} \g^\e, \dots, \partial_{\t_{d-1}} \g^\e \right)}{ \| \bigtimes\left( \partial_{\t_1} \g^\e, \dots, \partial_{\t_{d-1}} \g^\e \right) \|  },
\end{equation*}
for $\e\ge 0$, where $ \| \cdot  \|  $ denotes the Euclidean norm and $\bigtimes\left( \cdot \right)$ maps the arguments to a vector that is orthogonal to all the arguments. It therefore generalizes the cross product to higher dimensions and in particular, it is consistent with the cross product in $\R^3$. We denote it as \textit{extended cross product} and define it as in \cite[Section 1.1.4]{defnormal}:
\begin{definition}[Extended cross product]\label{def: extended cross product}
    Let $w_1, \dots, w_{d-1} \in \R^d$ be linearly independent. The extended cross product $\bigtimes\left( w_1,\dots,w_{d-1} \right)$ satisfies 
    \begin{equation*}
        \bigtimes\left( w_1,\dots,w_{d-1} \right) \cdot z = \operatorname{det}\left( w_1, \dots, w_{d-1}, z \right)
    \end{equation*}
    for any $z \in \R^d$.
\end{definition}
Inserting unit vectors for $z$ and using the Laplace expansion for the determinant we obtain a component-wise formula for the extended cross product: 
\begin{equation}\label{eq: expl cross}
    \left( \bigtimes\left( w_1,\dots,w_{d-1} \right) \right)_i = \left( -1 \right)^{i+d} \operatorname{det}\left( (w_1, \dots, w_{d-1})_{\hat{i}} \right)
\end{equation}
 where $( A )_{\hat{i}}$ describes the submatrix of a matrix $A \in \R^{d\times \left( d-1 \right)}$ that results from $A$ by removing the $i$-th row.
From this, we deduce
\begin{equation}\label{eq: nehat comp}
    \widehat{n^\e}_i = \left( -1 \right)^{i+d} \operatorname{det}\left( (\partial_{\t_1} \g^\e, \dots, \partial_{\t_{d-1}} \g^\e)_{\hat{i}} \right) = \left( -1 \right)^{i+d} \operatorname{det}\left( (D_\t\g^\e)_{\hat{i}} \right).
\end{equation}
The norm of $\widehat{n^\e}$ is simplified using the Cauchy-Binet formula, see e.g. \cite{BW89}:
\begin{equation}\label{eq: cauchy binet}
     \| \widehat{n^\e} \| = \sqrt{\operatorname{det}\left( (D_\t\g^\e)^T (D_\t\g^\e) \right)}.
\end{equation}
With this we define the normal as in \cref{eq:definition normal}. 
Next, we discuss the \textbf{expansion of the perturbed normal}.

With the assumed expansion in \cref{eq: gamma eps expansion}, we obtain
\begin{equation*}
    D_\t \g^\e \doteq D_\t \g^0 + \e D_\t(r n^0) = D_\t \g^0 + \e \left( n^0 \otimes \left( \nabla_\t r \right) + r D_\t n^0 \right),
\end{equation*}
where $\otimes$ describes the outer product.
The submatrix of this is written as
\begin{equation*}
    (D_\t \g^\e)_{\hat{i}} \doteq  (D_\t \g^0)_{\hat{i}} + \e \left( (n^0)_{\hat{i}} \otimes \left( \nabla_\t r \right) + r (D_\t n^0)_{\hat{i}} \right).
\end{equation*}
We use the derivative of the determinant, see e.g. \cite{MN19}, to expand each component of $\widehat{n^\e}$ in \cref{eq: nehat comp}: 
\begin{equation*}
    \begin{split}
        \widehat{n^\e}_i  &\doteq \left( -1 \right)^{i+d} \left( \operatorname{det}\left( (D_\t\g^0)_{\hat{i}} \right)  + \e \operatorname{tr}\Big( \operatorname{adj}\left(( D_\t \g^0  )_{\hat{i}}\right)  \big( (n^0)_{\hat{i}}  \otimes \left( \nabla_\t r \right) + r (D_\t n^0)_{\hat{i}} \big)\right)\Big) \\ 
        &= \widehat{n^0}_i + \e \left( -1 \right)^{i+d} \Big( \operatorname{tr}\left( \operatorname{adj}\left(( D_\t \g^0  )_{\hat{i}}\right) \left( (n^0)_{\hat{i}}  \otimes \left( \nabla_\t r \right) \right) \right) + r \operatorname{tr}\left( \operatorname{adj}\left(( D_\t \g^0  )_{\hat{i}}\right) (D_\t n^0)_{\hat{i}} \right) \Big),
    \end{split}
\end{equation*}
where we denote by $\operatorname{adj}(\cdot)$ the adjugate and by $\operatorname{tr}(\cdot)$ the trace. In the second equality we used the linearity of the trace.  Because of the symmetry of the trace, we have \\ $\operatorname{tr}\left( \operatorname{adj}\left(( D_\t \g^0  )_{\hat{i}}\right)  (n^0)_{\hat{i}} \otimes \left( \nabla_\t r \right)  \right) = \left( \nabla_\t r \right) \cdot \big(\operatorname{adj}\left(( D_\t \g^0  )_{\hat{i}}\right) (n^0)_{\hat{i}}\big)$, therefore
\begin{equation*}
    \begin{split}
        \widehat{n^\e}_i  &\doteq \widehat{n^0}_i + \e \left( -1 \right)^{i+d} \Big( \left( \nabla_\t r \right) \cdot \operatorname{adj}\left(( D_\t \g^0  )_{\hat{i}}\right) (n^0)_{\hat{i}} + r  \operatorname{tr}\left( \operatorname{adj}\left(( D_\t \g^0  )_{\hat{i}}\right) (D_\t n^0)_{\hat{i}} \right)\Big).
    \end{split}
\end{equation*}
Consequently, the first-order term for the expansion of $\widehat{n^\e}$ is given by 
\begin{equation*}
    \hat{\eta} := \Bigg(\left( -1 \right)^{i+d} \left( \left( \nabla_\t r \right) \cdot \operatorname{adj}\left(( D_\t \g^0  )_{\hat{i}}\right) (n^0)_{\hat{i}} + r  \operatorname{tr}\left( \operatorname{adj}\left(( D_\t \g^0  )_{\hat{i}}\right) (D_\t n^0)_{\hat{i}} \right)\right)\Bigg)_{i=1}^{d}. 
\end{equation*}
We rewrite this in a simplified form:
\begin{equation*}
    \hat{\eta} := H \left( \nabla_\t \r \right) +  h\r,
\end{equation*}
with $H \in \R^{d \times \left( d-1 \right)}$ and $h \in \R^d$ given as
\begin{equation}\label{eq: def H h}
    H \! \coloneq \! \Bigg(\!\! \left( -1 \right)^{i+d} \left( \operatorname{adj}\left(( D_\t \g^0  )_{\hat{i}}\right) (n^0)_{\hat{i}} \right)^T \!\! \Bigg)_{i=1}^{d} \!\text{and } h \!\coloneq \!\Bigg( \!\! \left( -1 \right)^{i+d} \operatorname{tr}\left( \operatorname{adj}\left(( D_\t \g^0  )_{\hat{i}}\right) (D_\t n^0)_{\hat{i}} \right) \!\!\Bigg)_{i=1}^{d}.
\end{equation}
Next, we discuss some relations of $H$ and $h$ to the surface $\operatorname{Im}(\g^0(\cdot,t))$, starting with $h$:\\
Since $0=\partial_{\t_i}\left( n^0 \cdot n^0 \right) = 2 n^0  \cdot \left(\partial_{\t_i} n^0 \right)$ for $i=1, \dots , d-1$, we know that every column of $D_\t n^0$ is tangential to $\operatorname{Im}(\g^0(\cdot,t))$ for every $t$. Since the columns of $D_\t\g^0$ are also tangential to $\operatorname{Im}(\g^0(\cdot,t))$, we can find a matrix $A \in \R^{\left( d-1 \right) \times \left( d-1 \right)}$ with 
\begin{equation*}
    D_\t n^0 = (D_\t\g^0) A \quad \text{and} \quad (D_\t n^0)_{\hat{i}} = (D_\t\g^0)_{\hat{i}} A.
\end{equation*}
Plugging this into \cref{eq: def H h}, we obtain 
\begin{equation*}
    h \!\coloneq \!\Bigg( \! \left( -1 \right)^{i+d} \operatorname{tr}\left( \operatorname{adj}\left(( D_\t \g^0  )_{\hat{i}}\right) (D_\t\g^0)_{\hat{i}} A \right) \!\Bigg)_{i=1}^{d} = \Bigg( \! \left( -1 \right)^{i+d} \operatorname{det}\left( (D_\t\g^0)_{\hat{i}} \right)\operatorname{tr}\left( A \right) \!\Bigg)_{i=1}^{d},
\end{equation*}
where we used the linearity of the trace and the identity $\operatorname{adj}(M)M=\operatorname{det}(M)I$ for\\ \mbox{$M \in \R^{\left( d-1 \right)\times \left( d-1 \right)}$} and $I$ the identity matrix.
From \cref{eq: nehat comp} we deduce
\begin{equation*}
    h = \widehat{n^0} \operatorname{tr}(A).
\end{equation*}
For $H$ we prove that every column is perpendicular to $n^0$.
First, we note that each entry of the matrix $H$ can be written as
\begin{equation*}
    H_{ik} = \left( -1 \right)^{i+d} \sum_{l=1}^{d-1} \operatorname{adj}\left(( D_\t \g^0  )_{\hat{i}}\right)_{kl} ((n^0)_{\hat{i}})_l.
\end{equation*}
To connect this to a geometric property of the surface, we first define $A^{\left( k \right)}\in \R^{d \times (d-1)}$ as the matrix that results from a matrix $A \in \R^{d \times (d-1)}$ by replacing the $k$-th column with $n^0$. Now consider the determinant of the submatrix of $((D_\t \g^0)^{\left( k \right)})_{\hat{i}}$ that can be rewritten using the Laplace expansion with regard to the $k$-th column: 
\begin{equation*}   
    \operatorname{det}\big(((D_\t \g^0)^{\left( k \right)})_{\hat{i}} \big) = \sum_{l=1}^{d-1}   \operatorname{adj}\left(( D_\t \g^0  )_{\hat{i}}\right)_{kl} \left( ( n^0 )_{\hat{i}} \right)_l. 
\end{equation*} 
Thus, we have 
\begin{equation*}
    H_{ik} = \left( -1 \right)^{i+d} \operatorname{det}\big(((D_\t \g^0)^{\left( k \right)})_{\hat{i}} \big) = \left( \bigtimes\left( \partial_{\t_1} \g^0, \dots,\partial_{\t_{k-1}} \g^0,n^0,\partial_{\t_{k+1}} \g^0 , \dots , \partial_{\t_{d-1}} \g^0 \right) \right)_i,
\end{equation*}
where we used the identity in \cref{eq: expl cross}.
Consequently, for every column of $H$, we obtain
\begin{equation*}
    H_{\cdot k } = \bigtimes\left( \partial_{\t_1} \g^0, \dots,\partial_{\t_{k-1}} \g^0,n^0,\partial_{\t_{k+1}} \g^0 , \dots , \partial_{\t_{d-1}} \g^0 \right), \quad k=1,...,d-1,
\end{equation*}
which is perpendicular to $n^0$ by \cref{def: extended cross product}. 

Finally, we discuss the expansion of the normalized $n^\e$ in \cref{eq:definition normal}. We have
\begin{equation}
    \begin{split}
        n^\e &\doteq \frac{\widehat{n^0} + \e \hat{\eta}}{ \| \widehat{n^0} + \e \hat{\eta} \|  } = \frac{\widehat{n^0}}{ \| \widehat{n^0} \|  } + \e \left( \frac{\hat{\eta}  \| \widehat{n^0} + \e \hat{\eta} \| - \left( \widehat{n^0} + \e \hat{\eta} \right) \left( \frac{\widehat{n^0} + \e \hat{\eta}}{ \| \widehat{n^0} + \e \hat{\eta} \|  } \cdot \hat{\eta} \right) }{{ \| \widehat{n^0} + \e \hat{\eta} \|^2  }} \Bigg \vert_{\e=0}\right) \\
        &= n^0 + \e \frac{1}{\| \widehat{n^0} \|}\left( \hat{\eta} - n^0 \left( n^0 \cdot \hat{\eta} \right) \right) =: n^0 + \e \eta.
    \end{split}
\end{equation}
Because of the relations that we discussed before, we know that $h$ is a multiple of $n^0$, which leads to 
\begin{equation*}
    h - n^0 (n^0 \cdot h)=0.
\end{equation*}
Furthermore, we know that each column of $H$ is perpendicular to $n^0$, which results in 
\begin{equation*}
    n^0 \cdot (H \left( \nabla_\t r \right)) = 0.
\end{equation*}
Thus, together with \cref{eq: cauchy binet}, we obtain for the first-order variation of $n^\e$: 
\begin{equation*}
    \eta= \frac{1}{\| \widehat{n^0} \|}\left( \hat{\eta} - n^0 \left( n^0 \cdot \hat{\eta} \right) \right) = \frac{1}{\sqrt{\operatorname{det}\left( (D_\t\g^0)^T (D_\t\g^0) \right)}} H \left( \nabla_\t r \right),
\end{equation*}
as defined in \cref{eq: def eta}.

\section{Appendix - Supplements to \cref{theorem}}

The following definition is a multi-dimensional extension of definitions in \cite{osti_482447,Bressan2000aa,HERTY2023473}. 
\begin{definition}[Broad solution]\label{broad solution}
\label{def:broad-solution-multidim}
    Consider the quasi–linear scalar partial differential equation
    \begin{equation}\label{eq:broad-sol-multidim}
        u_t(t,x) + a(t,x)\cdot \nabla_x u(t,x) = h(t,x,u),
        \qquad (t,x)\in [0,T]\times \mathbb{R}^d,
    \end{equation}
    where $a:[0,T]\times \mathbb{R}^d\to \mathbb{R}^d$ is Lipschitz continuous and
    $h:[0,T]\times \mathbb{R}^d\times \mathbb{R}\to \mathbb{R}$ is measurable with respect to
    $(t,x)$ and Lipschitz continuous with respect to $u$.
    Assume an initial condition $u(0,x)=u_0(x), u_0\in L^1(\mathbb{R}^d,\R)$. For $(\tau,\xi)\in [0,T]\times \mathbb{R}^d$, denote by $t\mapsto y(t;\tau,\xi)$ the solution to the Cauchy problem
    \begin{equation*}
        \frac{d}{dt}y(t)=a(t,y(t)), \qquad y(\tau)=\xi.
    \end{equation*}
    A locally integrable function $u\in L^1_{\mathrm{loc}}([0,T]\times \mathbb{R}^d,\R)$ fulfilling 
    \begin{equation*}
        \frac{d}{dt}u(t,y(t;\tau,\xi)) = h(t,y(t;\tau,\xi),u(t,y(t;\tau,\xi)))
    \end{equation*}
    is called a broad solution to \eqref{eq:broad-sol-multidim} if, for almost every
    $(\tau,\xi)\in [0,T]\times \mathbb{R}^d$, the following holds
    \begin{equation*}
        u(\tau,\xi) = u_0(y(0;\tau,\xi)) + \int_0^\tau h\bigl(s,y(s;\tau,\xi),u(s,y(s;\tau,\xi))\bigr)\,ds.
    \end{equation*}
\end{definition}

The following Corollary follows directly from \cref{theorem}
\begin{corollary}[Conservative form of evolution equation \eqref{eq:evolution r} for $\r$]\label{cor:conservative r eq}
    For $d=2$, we rewrite the evolution equation \eqref{eq:evolution r} for $\r$ in conservative form: 
    \begin{equation}
    \begin{split}
            {\partial_t \r(\t,t)} + \partial_\t\left( G(\t,t) { \r(\t,t)}  \right)= \tilde{G}(\t,t)  {\r(\t,t)} + \hat{g}(\t,t),  \quad  {r}(\cdot,0 )= {r_0}(\cdot ),
    \end{split}
\end{equation}
    where $G$, $\tilde{G}$ and $\hat{g}$ are given by
    \begin{equation}
    \begin{split}
        G& \coloneqq  \frac{1}{ \| \partial_\t \g^0 \|^2 } \frac{[f(\un)]}{[\un]} \cdot \left( \partial_\t \g^0 \right), \\
        \tilde{G}& \coloneqq  \frac{n^0}{[\un]^2} \cdot \Big( [\un]\left[ f'(\un) \left( (\nabla_x \un) \cdot n^0 \right) \right] - \left[ f(\un)\right] \left( [\nabla_x \un] \cdot n^0 \right) \Big)\\
        &\hspace{0.6cm} + \frac{1}{ \| \partial_\t \g^0 \|^2 }\left( n^0 \cdot \frac{[f(\un)]}{[\un]}\right) \left( n^0 \cdot \left( \partial_\t^2 \g^0 \right) \right)\\
        &\hspace{0.6cm}+ \frac{1}{[\un]^2} \frac{\partial_\t \g^0}{\| \partial_\t \g^0 \|^2} \cdot \left( [f'(\un) \left( (\nabla_x \un) \cdot \left( \partial_\t \g^0 \right) \right)] [\un]- [f(\un)]\left( [\nabla_x \un] \cdot \left( \partial_\t \g^0 \right) \right) \right)\\
        &\hspace{0.6cm} - \frac{1}{ \| \partial_\t \g^0 \|^4 }\left( \left( \partial_\t \g^0 \right) \cdot \frac{[f(\un)]}{[\un]}  \right) \left( \left( \partial_\t \g^0 \right)\cdot \left( \partial_\t^2 \g^0 \right) \right),\\
        \hat{g}& \coloneqq  \frac{n^0}{[\un]^2} \cdot \left( [\un] [f'(\un) {v}] - [f(\un)] [{v}] \right).
    \end{split}
    \end{equation}
\end{corollary}

\section{Proof of \cref{prop: lec_zuazua}}\label{sec: equivalence}
 
In the two-dimensional setting, \cref{eq:evolution r} and \cref{eq:evolution r: functions} reduce to:

    \begin{equation}\label{eq: 2d nonconservative}
        \begin{split}
            {\partial_t \r(\t,t)} +  g(\t,t)  (\partial_\t \r(\t,t))  - \tilde{g}(\t,t)  {\r(\t,t)} - \hat{g}(\t,t)=0,  \quad  {r}(\cdot,0 )= {r_0}(\cdot ),
        \end{split}
    \end{equation}
    with
    \begin{equation}\label{eq: 2d nonconservative functions}
        \begin{split}
            g& \coloneqq  \frac{1}{[\un] \| \partial_\t \g^0 \|^2 } [f(\un)]\cdot (\partial_\t \g^0) , \\
            \tilde{g}& \coloneqq  \frac{n^0}{[\un]^2} \cdot \Big( [\un]\left[ f'(\un) \left( \nabla_x \un \cdot n^0 \right) \right] - \left[ f(\un)\right] \left( [\nabla_x \un] \cdot n^0 \right) \Big),\\
            \hat{g}& \coloneqq  \frac{n^0}{[\un]^2} \cdot \left( [\un] [f'(\un) {v}] - [f(\un)] [{v}] \right).
        \end{split}
    \end{equation}

    We give the proof for \cref{prop: lec_zuazua} using the \cref{lemma: equivalence} below.
    \begin{proof}[Proof of \cref{prop: lec_zuazua}]
The equivalence between \cref{eq:evolution v} and the corresponding
equation in \cite[Theorem 4.5]{LZ16} is straightforward. Hence, we
focus on the evolution equation for $\r$.

To prove the equivalence of \cite[Eq.~(4.17)]{LZ16} and \cref{eq: 2d nonconservative}, we first note that Lecaros and Zuazua consider the shock surface as the graph in space-time, which in
our notation is given by $(\g^0(\t,t),t) \in \R^3$ for $\t\in\R$ and
$t\in[0,T]$. For their choice of parametrization, the spatial velocity
$\partial_t \g^0$ is not required to be orthogonal to the spatial shock
curve. Hence, in general, it may contain a tangential component.
Therefore, terms of the form $\left( \partial_t \g^0 \right) \cdot \frac{\partial_\t \g^0}{ \| \partial_\t \g^0 \|}$ need not vanish in their framework. 

In our paper, we fix the parametrization freedom by choosing the time derivative of the parametrization to have no tangential component, i.e. 
\begin{equation*}
    \left( \partial_t \g^0 \right) \cdot \frac{\partial_\t \g^0}{ \| \partial_\t \g^0 \|} =0,
\end{equation*}
see \cref{eq: shockspeeds,eq: RH}.
Both choices are able to describe the same geometric object, namely a discontinuity curve separating the spatial domain into two subdomains. They only differ in the tangential parametrization of this curve.
Note also, that their spatial normal vector is oriented as $-n^0$, which reverses the signs of $\r$ and of the jumps $[\cdot]$.

    Applying \cref{lemma: equivalence} yields the equivalence between \cite[Eq.~(4.17)]{LZ16} and \cref{eq: 2d nonconservative}.
\end{proof}

    \begin{lemma}[Equivalence of Evolution Equations]\label{lemma: equivalence}
        Assume that the evolution of the parametrization of $\g^0$ is chosen such that
$\partial_t \g^0$ is purely normal, i.e. $\left( \partial_t \g^0 \right) \cdot \frac{\partial_\t \g^0}{ \| \partial_\t \g^0 \|}=0$. Then, \cite[Eq.~(4.17)]{LZ16} and \cref{eq: 2d nonconservative} are equivalent, i.e.
        \begin{equation}\label{eq: lemma eq}
            \begin{split}
                &\frac{1}{\| \partial_\t \g^0 \|} \left( \partial_\t \left( (-B) (-r) \right) + \partial_t \left( \| \partial_\t \g^0 \| (-[\un]) (-r) \right)\right) - \left( -[f'(\un) v],-[v] \right) \cdot \left( {-n^0}, {s^0} \right)=0\\
                    &\hspace{0.5cm}=
                    {\partial_t \r(\t,t)} +  g(\t,t)  (\partial_\t \r(\t,t))  - \tilde{g}(\t,t)  {\r(\t,t)} - \hat{g}(\t,t),
            \end{split}
        \end{equation}
        where we inserted \cite[Eq.~(4.37)]{LZ16} with $B=\left( \left( \partial_t\g^0,0 \right) \cdot (\tau^0,0) \right)[\un] +  ([f(\un)],0) \cdot ({\tau^0},0) $ and the spatial tangential ${\tau^0} := \frac{\partial_\t \g^0}{ \| \partial_\t \g^0 \|  }$.
    \end{lemma}
    \begin{proof}
        Note that when rewriting \cite[Eq.~(4.17)]{LZ16} in our notation, we need a factor $\frac{1}{\sqrt{1+(s^0)^2}}$, which we already canceled on the left-hand side of \cref{eq: lemma eq}.
        Since the time derivative is only dependent on the normal, we have that $\left( \partial_t \g^0,0 \right) \cdot \left( {\tau^0},0 \right)=0$, therefore $B= [f(\un)] \cdot {\tau^0} $. 
        We simplify the left-hand side of \cref{eq: lemma eq}: 
        \begin{equation*}
        \begin{split}
             \frac{1}{\| \partial_\t \g^0 \|} \left( \partial_\t \left( B r \right) + \partial_t \left( \| \partial_\t \g^0 \| [\un] r \right)\right) 
            = [f'(\un) v] \cdot n^0 - [v] s^0 .
        \end{split}
    \end{equation*}
    We can expand this to obtain 
    \begin{equation}\label{eq: inbetween comparison}
        \begin{split}
            \partial_t \r + \frac{B}{[\un]  \| \partial_\t \g^0 \|  } \partial_\t \r  = &-\frac{1}{[\un]  \| \partial_\t \g^0 \|}\left( (\partial_t [\un])  \| \partial_\t \g^0 \| + (\partial_t  \| \partial_\t \g^0 \|) [\un] + \partial_\t B \right) \r \\
            &+ \frac{[f'(\un) v] \cdot n^0 - [v] s^0}{[\un]}.
        \end{split}
    \end{equation}
    A short comparison of \cref{eq: inbetween comparison} with \cref{eq: 2d nonconservative} shows $\frac{B}{[\un]  \| \partial_\t \g^0 \|  } = g$ and  $\frac{[f'(\un) v] \cdot n^0 - [v] s^0}{[\un]}= \hat{g}$. Thus, we only need to show 
    \begin{equation}\label{eq: tildeg}
        \tilde{g}=-\frac{1}{[\un]  \| \partial_\t \g^0 \|}\left( (\partial_t [\un])  \| \partial_\t \g^0 \| + (\partial_t  \| \partial_\t \g^0 \|) [\un] + \partial_\t B \right) =: (\ast).
    \end{equation}
     For this purpose, we first show four auxiliary statements:
    \begin{enumerate}
        \item[$*_{(1)}$ :] $\partial_t  \| \partial_\t \g^0 \| \stackrel{\eqref{eq: RH}}{=} \frac{\partial_\t \g^0}{ \| \partial_\t \g^0 \|  } \cdot \partial_\t( s^0 n^0 ) = s^0 {\tau^0}  \cdot \left( \partial_\t n^0 \right) $ and since $\partial_\t n^0 = \rotmat \partial_\t {\tau^0}$, \cref{eq: shockspeeds}, and $a\cdot \left( \rotmat b \right) = -\left( \rotmat a \right) \cdot b $, we obtain $\partial_t  \| \partial_\t \g^0 \| = - \frac{[f(\un)] \cdot n^0 }{[\un]} \left( n^0 \cdot \left( \partial_\t {\tau^0} \right) \right)$
        \item[$*_{(2)}$ :] $\partial_\t B = (\partial_\t [f(\un)]) \cdot {\tau^0} + [f(\un)] \cdot \left( \partial_\t {\tau^0} \right) = {\tau^0} \cdot \left[f'(\un) \left( (\nabla_x \un) \cdot \left( \partial_\t \g^0 \right)  \right)\right] + [f(\un)] \cdot \left( \partial_\t {\tau^0} \right)$
        \item[$*_{(3)}$ :] $\partial_\t {\tau^0} = \kappa n^0$ for a $\kappa \in \R$, since $0=\partial_\t   \| {\tau^0} \|^2 = 2\left( {\tau^0} \cdot \left( \partial_\t {\tau^0} \right) \right)  $
        \item[$*_{(4)}$ :] We split the gradient of $\un$ into tangential and normal components with respect to $\g^0$: $\nabla_x \un = {\tau^0} \partial_{\tau^0} \un + n^0 \partial_{n^0} \un$, therefore, 
        \begin{enumerate}
            \item[$*_{(4.1)}$ :] $\partial_{n^0} \un = (\nabla_x \un) \cdot n^0$
            \item[$*_{(4.2)}$ :] with $\partial_t \un(\g^0(\t,t),t) = \partial_t \un + (\partial_t \g^0) \cdot (\nabla_x \un) \stackrel{\eqref{eq: IVP},\eqref{eq: RH}}{=} -f'(\un) \cdot \left( \nabla_x \un \right) + s^0 \underbrace{n^0 \cdot \left( \nabla_x \un \right)}_{\stackrel{*_{(4.1)}}{=} \partial_{n^0}\un}$, we obtain $\partial_t [\un] = -[f'(\un) \cdot {\tau^0} \partial_{\tau^0} \un ] + [s^0 \partial_{n^0} \un - f'(\un) \cdot n^0 \partial_{n^0}\un]$
        \end{enumerate}
    \end{enumerate}
    Now we can show that \cref{eq: tildeg} is true: 
    \begin{equation*}
        \begin{split}
            &(\ast)\stackrel{*_{(1)},*_{(2)}}{=} -\frac{1}{[\un]  \| \partial_\t \g^0 \|}\bigg( (\partial_t [\un])  \| \partial_\t \g^0 \| + \left( - \frac{[f(\un)] \cdot n^0 }{[\un]} \left( n^0 \cdot \left( \partial_\t {\tau^0} \right) \right) \right) [\un] \\
            &\hspace{4.2cm}+ {\tau^0} \cdot \left[f'(\un) \left( (\nabla_x \un) \cdot \left( \partial_\t \g^0 \right)  \right)\right] + [f(\un)] \cdot \left( \partial_\t {\tau^0} \right) \bigg) \\ 
            &= -\frac{1}{[\un]  \| \partial_\t \g^0 \|}\bigg( (\partial_t [\un])  \| \partial_\t \g^0 \| + {\tau^0} \cdot \left[f'(\un) \left( (\nabla_x \un) \cdot \left( \partial_\t \g^0 \right)  \right)\right]\\
            &\hspace{3.5cm}+ \underbrace{\left( [f(\un)] - n^0 \left( [f(\un)] \cdot n^0 \right) \right)\cdot \left( \partial_\t {\tau^0} \right)}_{\stackrel{*_{(3)}}{=}0} \bigg) \\ 
            &\!\stackrel{*_{(4.2)}}{=} -\frac{1}{[\un]  \| \partial_\t \g^0 \|}\bigg(  [s^0 \partial_{n^0} \un - f'(\un) \cdot n^0 \partial_{n^0}\un]\| \partial_\t \g^0 \| \\
            &\hspace{3.5cm}\underbrace{-[f'(\un) \cdot \left( \partial_\t \g^0 \right) \partial_{\tau^0} \un ] + {\tau^0} \cdot \big[f'(\un) \underbrace{\left( (\nabla_x \un) \cdot \left( \partial_\t \g^0 \right)  \right)}_{\stackrel{*_{(4)}}{=} {\tau^0} (\partial_{\tau^0} \un) \cdot \left( \partial_\t \g^0 \right) = \| \partial_\t \g^0 \|(\partial_{\tau^0} \un)}\big] }_{=0}\!\bigg) \\ 
            &= \frac{1}{[\un]}\left([f'(\un) \cdot n^0 \partial_{n^0}\un] - [s^0 \partial_{n^0} \un ] \right).
        \end{split}
    \end{equation*}
    Finally, using $*_{(4.1)}$ and \cref{eq: shockspeeds}, we arrive at 
    \begin{equation*}
        \frac{1}{[\un]}\left(\big[f'(\un) \cdot n^0 \left( (\nabla_x \un) \cdot n^0 \right)\big] - \frac{[f(\un)]\cdot n^0}{[\un]} ([\nabla_x \un ]\cdot n^0)\right) = \tilde{g}.
    \end{equation*}
    \end{proof}

\end{document}